\documentclass[12pt]{amsart}
\usepackage{amssymb}
\usepackage{amsmath}
\usepackage{amsbsy}
\usepackage{amscd}
\advance\textheight by \topskip
\makeatletter
\def\cal{\mathcal}
\def\Bbb{\mathbb}

\newenvironment{pf*}[1]{\proof[#1]}{\endproof}
\newcommand{\rom}{\textup}
\renewcommand{\thesubsection}{\thesection(\@roman\c@subsection)}
\makeatother
\newtheorem{Theorem}[equation]{Theorem}

\newtheorem{Lemma}[equation]{Lemma}
\newtheorem{Proposition}[equation]{Proposition}

\theoremstyle{definition}
\newtheorem{Definition}[equation]{Definition}

\makeatletter
\renewcommand\section{\@startsection{section}{1}%
  {\z@}{.7\linespacing\@plus\linespacing}{.5\linespacing}%
  {\reset@font\normalfont\bfseries\centering}}
\makeatother

\theoremstyle{remark}
\newtheorem{Remark}[equation]{Remark}

\newtheorem*{Acknowledgements}{Acknowledgements}
\numberwithin{equation}{section}
\numberwithin{figure}{section}

\newcommand{\thmref}[1]{Theorem~\ref{#1}}
\newcommand{\secref}[1]{\S\ref{#1}}
\newcommand{\lemref}[1]{Lemma~\ref{#1}}
\newcommand{\propref}[1]{Proposition~\ref{#1}}
\newcommand{\remref}[1]{Remark~\ref{#1}}

\newcommand{\subsecref}[1]{\S\ref{#1}}

\newcommand{\Romnum}[1]{\expandafter\uppercase\expandafter{\romannumeral #1}} 
\newcommand{\C}{{\Bbb C}}
\newcommand{\Z}{{\Bbb Z}}

\newcommand{\R}{{\Bbb R}}
\newcommand{\HH}{{\Bbb H}}
\newcommand{\RP}{\operatorname{\R P}}
\newcommand{\SO}{\operatorname{\rm SO}}
\newcommand{\U}{\operatorname{\rm U}}
\newcommand{\Pin}{\operatorname{\rm Pin}}
\newcommand{\Spin}{\operatorname{\rm Spin}}
\newcommand{\Spinc}{\Spin^{c}}
\newcommand{\SP}{\operatorname{\rm Sp}}

\newcommand{\Map}{\operatorname{Map}}

\newcommand{\Hom}{\operatorname{Hom}}
\newcommand{\Ext}{\operatorname{Ext}}

\newcommand{\im}{\mathop{\text{\rm im}}\nolimits}
\newcommand{\ad}{\operatorname{ad}}
\newcommand{\rank}{\operatorname{rank}}
\newcommand{\sign}{\operatorname{sign}}

\newcommand{\real}{\operatorname{Re}}
\newcommand{\ind}{\mathop{\text{\rm ind}}\nolimits}
\newcommand{\tr}{\operatorname{tr}}

\newcommand{\M}{{\cal M}} % the moduli space
\newcommand{\G}{{\cal G}} % the gauge transformation group

\newcommand{\A}{{\cal A}} % the space of connections
\newcommand{\CC}{{\cal C}} % the configuration space
 
\newcommand{\WW}{{\cal W}} % G-universe
\newcommand{\VV}{{\cal V}} % Vector bundle
\newcommand{\OO}{{\operatorname{O}}} %
\newcommand{\K}{{\cal K}} %
\newcommand{\muh}{\hat{\mu}}% mu hat

\newcommand{\Lie}{\operatorname{Lie}}

\newcommand{\X}{\tilde{X}}
\newcommand{\ct}{\tilde{c}}
\begin{document}
\title[$\Pin^-(2)$-monopole equations and intersection forms with local coefficients]{$\Pin^-(2)$-monopole equations and intersection forms with local coefficients of $4$-manifolds}
%\titlerunning{$\Pin^-(2)$-monopole equations and intersection forms with local coefficients}
% 
%
%
%\journalname{}
%\date{}
\author{Nobuhiro Nakamura}
\address{Graduate School of Mathematical Sciences, University of Tokyo, 3-8-1, Komaba, Meguro-ku, Tokyo, 153-8914, Japan}
\email{nobuhiro@ms.u-tokyo.ac.jp}
%
%\maketitle
%
\begin{abstract}
We introduce a variant of the Seiberg-Witten equations, $\Pin^-(2)$-monopole equations, and give its applications to intersection forms with local coefficients of $4$-manifolds.
The first application is an analogue of Froyshov's results on $4$-manifolds with definite intersection forms with local coefficients.
The second is a local coefficient version of Furuta's $10/8$-inequality.
As a corollary, we construct nonsmoothable spin $4$-manifolds satisfying Rohlin's theorem and the $10/8$-inequality.
\end{abstract}
\keywords{4-manifold, intersection form, Seiberg-Witten equations}
%subclass[2010]{57R57}
\subjclass[2010]{57R57}
%57R57 Applications of global analysis to structures on manifolds, Donaldson and Seiberg-Witten invariants 
%
%
%
%
\maketitle
%
%
%
%
%%%%%%%%%%%%%%%%%%%%%%%%%%%%%%%%%%%%%%%%%%%%%%%%%%%%%%%%%%%%%%%%%%%%%%%%%%%%%%%
\section{Introduction}\label{sec:intro}
%%%%%%%%%%%%%%%%%%%%%%%%%%%%%%%%%%%%%%%%%%%%%%%%%%%%%%%%%%%%%%%%%%%%%%%%%%%%%%%
%
%

K.~Froyshov \cite{Froyshov} recently proved theorems on intersection forms with local coefficients of $4$-manifolds which can be considered as a local coefficient analogue of Donaldson's theorem for definite $4$-manifolds \cite{D1,D2}.
To prove his results, he analyzes the moduli space of $\SO(3)$-instantons, and effectively make use of the existence of a kind of reducibles, {\it twisted reducibles}, whose stabilizers are $\Z/2$, in order to extract the information on local coefficient cohomology.

The first part of this paper proves an analogue of Froyshov's results by Seiberg-Witten theory.
In fact, we prove that, if a closed smooth $4$-manifold has a definite intersection form with local coefficient, it should be the standard form.

To state the precise statement, we give some preliminaries.
Let $X$ be a closed, connected, oriented smooth $4$-manifold. 
Suppose a double covering $\X$ of $X$ is given. 
Let $l=\X\times_{\{\pm 1\}}\Z$ and $\lambda=\X\times_{\{\pm 1\}}\R$ be its associated bundles with fiber $\Z$ and $\R$.
We can consider the cohomology $H^*(X;l)$ with $l$ as bundle of coefficients.
Since $l\otimes l=\Z$, we have a homomorphism by the cup product, 
$$
H^2(X;l)\otimes H^2(X;l) \to H^4(X;\Z)=\Z.
$$
This induces a unimodular quadratic form $Q_{X,l}$ on $H^2(X;l)/\text{torsion}$.
Let $b_q(X;l)$ be the $l$-coefficient $q$-th Betti number, i.e., 
$$b_q(X;l) = \rank H^q(X;l)/\text{torsion}.$$ % = \rank H^q(X;\underline{\lambda})$.
The ordinary  $\Z$-coefficient Betti numbers are denoted by  $b_q(X)$.
The short exact sequence of bundles,
$$
0\to l\overset{\cdot 2}{\to} l\to\Z/2\to 0,
$$ 
induces a long exact sequence,
$$
\cdots\to H^q(X;l)\overset{\cdot 2}{\to}  H^q(X;l)\to  H^q(X;\Z/2)\to  H^{q+1}(X;l)\to \cdots.
$$
In particular, mod $2$ reduction map $H^2(X;l)\to H^2(X;\Z/2)$ is defined.

Our first theorem is as follows:
\begin{Theorem}\label{thm:main}
Let $X$ be a closed, connected, oriented smooth $4$-manifold. 
Suppose that a nontrivial $\Z$-bundle $l\to X$ satisfies the following{\rom :}
\begin{enumerate}
\item The intersection form $Q_{X,l}$ is definite.
\item Let $\lambda = l\otimes\R$. Then $w_1(\lambda)^2$ has a lift in the torsion part of $H^2(X;l)$.
\end{enumerate}
Then $Q_{X,l}$ is isomorphic to the diagonal form.
\end{Theorem}

The proof of \thmref{thm:main} is outlined as follows.
For the double covering $\X$ associated with $l$, let $\iota\colon\X\to\X$ be the covering transformation.
We consider a $\Spin^c$-structure $\tilde{c}$ on $\X$ together with an isomorphism (of order $4$) between the pullback $\Spin^c$-structure $\iota^*\tilde{c}$ and the complex conjugation of $\tilde{c}$.
In fact, if we start from a $\Spin^{c_-}$-structure on $X$, a $\Pin^-(2)$-variant of  $\Spin^c$-structure introduced in \secref{sec:spin-p}, we obtain an antilinear involution $I$ covering $\iota$ on the spinor bundles and the determinant line bundle of $\tilde{c}$.
Then, $I$ acts on the Seiberg-Witten moduli space $\tilde{\M}$ of $(\X,\tilde{c})$, and we pay attention to its $I$-invariant part $\tilde{\M}^I$.
In fact, on the $\Spin^{c_-}$-structure on $X$, we can define a variant of Seiberg-Witten equations, {\it $\Pin^-(2)$-monopole equations} we call, and we can identify the moduli space of solutions of the $\Pin^-(2)$-monopole equations, $\M$, with the $I$-invariant Seiberg-Witten moduli space   $\tilde{\M}^I$.
The rest of the argument is analogous to the argument in the alternative proof of Donaldson's theorem by the Seiberg-Witten theory (see e.g. \cite{Moore,Nicolaescu}).
That is, under the assumptions of \thmref{thm:main}, we prove the virtual dimension of $\M\cong\tilde{\M}^I$ cannot be greater than $b_1(X;l)$, and obtain an inequality for the characteristic elements of $Q_{X,l}$. 
Finally, we invoke a theorem of Elkies \cite{Elkies} to prove the form should be the standard form.

In the second part of the paper, the technique of finite dimensional approximation due to Furuta and Bauer \cite{Furuta,BF} is applied to  the {\it $\Pin^-(2)$-monopole map}, and we prove a $10/8$-type inequality for intersection forms with local coefficients:
\begin{Theorem}\label{thm:10/8}
Let $X$ be a closed connected oriented smooth $4$-manifold.
For any nontrivial $\Z$-bundle $l$ over $X$ which satisfies $w_1(\lambda)^2 = w_2(X)$, the following inequality holds{\rom :} 
$$b_+(X;l)\geq -\frac{\sign(X)}{8}.$$
\end{Theorem}
\begin{Remark}\label{rem:10/8}
(1) In the proof of the $10/8$-inequality by Furuta \cite{Furuta}, the existence of an extra $\Pin^-(2)$-action on the Seiberg-Witten theory on the $\Spin^c$-structure associated with a spin structure plays an essential role.
Analogously, a key point of the proof of \thmref{thm:10/8} is the existence of an extra gauge symmetry.
In fact,  there is a larger gauge symmetry on the $\Spin^{c_-}$-structure whose associated $\OO(2)$-bundle is $\underline{\R}\oplus \lambda$ (\subsecref{subsec:G}), and such a $\Spin^{c_-}$-structure exists if the condition $w_1(\lambda)^2 = w_2(X)$ is satisfied (\propref{prop:spin-p}). \\
(2) Note that $\alpha\cup \alpha = Sq^1(\alpha)$ for $\alpha\in H^1(X;\Z/2)$, and $Sq^1$ is the Bockstein connecting homomorphism associated with coefficient sequence
\begin{equation}\label{eq:coef}
0\to \Z/2\to\Z/4\to\Z/2\to 0.
\end{equation}
(\cite{Switzer}, 18.12.)
 For instance, if $w_2(X)$ has an integral lift of order $2$, then $w_2(X) = \alpha\cup \alpha $ holds for some $\alpha$.
This follows from comparing the Bockstein sequence associated with \eqref{eq:coef} with another Bockstein sequence associated with the sequence
$$
0\to \Z\overset{\times 2}{\to}\Z\to\Z/2\to 0.
$$
(3) As mentioned above, the proof of \thmref{thm:10/8} use the $\Pin^-(2)$-monopole map.
In fact, the $\Pin^-(2)$-monopole map can be considered as the $I$-invariant part of the Seiberg-Witten map of the double covering $\X$. 
Therefore, we can prove  \thmref{thm:10/8} by applying the finite dimensional approximation technique directly to the Seiberg-Witten equations on $\X$ with the $I$-action.\\  
(4) We will give an alternative proof of \thmref{thm:main} by using the same technique used in the proof of \thmref{thm:10/8}.
\end{Remark}

As an application of \thmref{thm:main} and \thmref{thm:10/8}, we construct nonsmoothable $4$-manifolds satisfying known constraints on smooth $4$-manifolds.

Let us consider the spin cases.
For smooth spin $4$-manifolds, we know two fundamental theorems, Rohlin's theorem(see e.g.\cite{Kirby}) and  Furuta's theorem \cite{Furuta}.
Rohlin's theorem tells us that the signature of every closed spin $4$-manifold is divisible by $16$. 
On the other hand, Furuta's theorem \cite{Furuta} tells us that every closed smooth spin $4$-manifold $X$ with indefinite form satisfies the so-called ``$10/8$-inequality''
$$
b_2 (X) \geq \frac54 |\sign(X)| + 2.
$$
This inequality is improved by M.~Furuta and Y.~Kametani \cite{FK} in the case when $b_1(X)>0$. 
We call the improved inequality in \cite{FK}  the strong $10/8$-inequality.
\begin{Theorem}\label{thm:spin}
There exist nonsmoothable closed spin topological $4$-manifolds  which have signatures divisible by $16$ and satisfy the strong $10/8$-inequality.
\end{Theorem}

The idea of the construction of such nonsmoothable examples is as follows.
Let $V$ be any simply-connected topological $4$-manifold with even definite form $Q_V$ of rank $16k$, and let $X$ be a connected sum of $V$ with sufficiently many $T^2\times S^2$'s or $T^4$'s so that the $10/8$-inequality is satisfied.
Since $b_2(M;l)=0$ and $w_1(\lambda)^2=0$ for a non-trivial $\Z$-bundle $l$ on $M=T^2\times S^2$ or $T^4$, we can show that $X$ is nonsmoothable by \thmref{thm:main}.
We can also construct similar examples by using \thmref{thm:10/8}.

C.~Bohr~\cite{Bohr} and Lee-Li~\cite{LL} proved $10/8$-type inequalities for  non-spin $4$-manifolds with even forms.
We also construct nonsmoothable non-spin $4$-manifolds with even forms satisfying their inequalities.
\begin{Theorem}\label{thm:non-spin}
There exist nonsmoothable closed non-spin $4$-manifolds $X$ with even indefinite forms satisfying $b_2(X)\geq \frac54|\sign(X)|.$
\end{Theorem}
\begin{Remark}
One of the results of Bohr~\cite{Bohr} and Lee-Li~\cite{LL} is that the inequality $b_2(X)\geq 5/4|\sign(X)|$ holds for non-spin $4$-manifolds $X$ with even indefinite forms whose $2$-primary torsion part of $H_1(X;\Z)$ is isomorphic to $\Z/2^k$ or $\Z/2\oplus\Z/2$.
We construct our examples so that the  $2$-primary torsion part of $H_1(X;\Z)$ is $\Z/2$.
\end{Remark}

The organization of the paper is as follows.
In Section 2, we prove \thmref{thm:spin} and \thmref{thm:non-spin} assuming \thmref{thm:main} and \thmref{thm:10/8}.
In Section 3, we introduce the notion of $\Spin^{c_-}$-structures which is a $\Pin^-(2)$-variant of $\Spin^c$-structures.
It is also explained that, if a $\Spin^{c_-}$-structure on $X$ is given, then a $\Spin^c$-structure on the double covering $\X$ is induced, and the covering transformation of $\X$ is covered by antilinear involutions $I$ on the spinor bundles and the determinant line bundle.
In Section 4, we introduce $\Pin^-(2)$-monopole equations, and show that the moduli space of solutions of $\Pin^-(2)$-monopole equations can be identified with the $I$-invariant Seiberg-Witten moduli space on the double covering $\X$.
We also analyze the structure of $\Pin^-(2)$-monopole moduli spaces when $b_+(X;\lambda)=0$.
In Section 5, we prove \thmref{thm:main}.
In Section 6, the Bauer-Furuta theory \cite{Furuta,BF} of $\Pin^-(2)$-monopole map is studied, and \thmref{thm:10/8} is proved by using the equivariant $K$-theory as in \cite{Furuta,Bryan}.
We also give an alternative proof of \thmref{thm:main} by the same technique.
\begin{Acknowledgements}
The author would like to express his deep gratitude to the referee for his detailed and valuable comments including a long list of suggestions over 40 items which enable the author to improve  the paper drastically.
It is also a pleasure to thank M.~Furuta, Y.~Kametani, K.~Kiyono and S.~Matsuo for helpful discussions and their comments on the earlier versions of the paper.
\end{Acknowledgements}
%
%
%
%%%%%%%%%%%%%%%%%%%%%%%%%%%%%%%%%%%%%%%%%%%%%%%%%%%%%%%%%%%%%%%%%%%%%%%%%%%%%%%
\section{Applications}\label{sec:applications}
%%%%%%%%%%%%%%%%%%%%%%%%%%%%%%%%%%%%%%%%%%%%%%%%%%%%%%%%%%%%%%%%%%%%%%%%%%%%%%%
%
%
%
%

In this section, we prove \thmref{thm:spin} and \thmref{thm:non-spin} assuming \thmref{thm:main} and \thmref{thm:10/8}.
First, we prove the following.
({\it Cf.} \cite{Froyshov}, Corollary 1.1.)
\begin{Theorem}\label{thm:nonsmooth}
Let $V$ be any closed oriented topological $4$-manifold which satisfies either of the following{\rm :}
\begin{enumerate}
\item the intersection form $Q_V$ on $H^2(V;\Z)$ is non-standard definite, or 
\item there exists an element $\alpha\in H^1(V;\Z/2)$ so that $\alpha\cup \alpha =w_2(V)$, and the intersection form $Q_{V,l_\alpha}$ satisfies $b_+(V;l_\alpha) < -\sign(V)/8$, where $l_\alpha$ is the $\Z$-bundle corresponding to $\alpha$.
{\rm (}If $w_2(V)=0$, then $\alpha$ may be $0$.{\rm )}
\end{enumerate}
Let $M$ be a closed oriented $4$-manifold which admits a nontrivial $\Z$-bundle $l^\prime \to M$ such that $b_2(M;l^\prime)=0$ and $w_1(\lambda^\prime)^2=0$, where $\lambda^\prime = l^\prime\otimes\R$.
Then the connected sum $X=V\# M$ does not admit any smooth structure.
\end{Theorem}

Before proving \thmref{thm:nonsmooth}, we will discuss how to construct $V$ and $M$ as in the theorem.
One can construct simply-connected examples of $V$ satisfying (1) by Freedman's theory \cite{FQ}.
Examples of $V$ satisfying (2) can be constructed as follows.
Let $|E_8|$ be the simply-connected topological $4$-manifold whose form is $-E_8$. 
(This can be  also constructed by Freedman's theory.)
Then $V=m|E_8|\# n(S^2\times S^2)$ with $m>n$ are spin manifolds satisfying (2) with $\alpha=0$.

As shown in Hambleton-Kreck's paper \cite{HK}(Proof of Theorem 3), there exist non-spin topological rational homology $4$-spheres $\Sigma_0$ and $\Sigma_1$ with $\pi_1=\Z/2$ and Kirby-Siebenmann obstructions ${\rm ks}(\Sigma_0)=0$ and ${\rm ks}(\Sigma_1)\neq 0$.
%(A smooth $\Sigma_0$ can be constructed as $\Sigma_0 = S^2\times S^2 /(x,y)\sim(-x,-y)$.)
For instance, an Enriques surface is topologically decomposed into $|E_8|\#(S^2\times S^2)\#\Sigma_1$. 
Then $V=m|E_8|\# n(S^2\times S^2) \#\Sigma_i$ with $m>n+1$ are non-spin manifolds satisfying (2) with non-zero class $\alpha\in H^1(V;\Z/2)\cong H^1(\Sigma_i;\Z/2)\cong\Z/2$ as follows.
First, note that $b_+(V;l_\alpha) = b_+(V)+1$ in this case.
This follows from the following fact: for any $\Z$-bundle $l$ over a manifold $X$, let $\X$ be the double covering corresponding to $l$, and let $\underline{\lambda}=l\otimes \R$ considered as a bundle with discrete fibers. 
Then, we have in general,
\begin{gather*}
b_0(X)-b_1(X)+b_+(X) = b_0(X;l)-b_1(X;l)+b_+(X;l), \\
H^*(\X;\R) = H^*(X;\R)\oplus H^*(X;\underline{\lambda}).
\end{gather*}
Second, since $H^2(\Sigma_i;\Z)=\Z/2$, the integral lift of $w_2(\Sigma_i)$ has order $2$. 
By \remref{rem:10/8}(2), the generator $\alpha\in H^1(\Sigma_i;\Z/2)=\Z/2$ should satisfy $w_2(\Sigma_i)=\alpha\cup\alpha$.
Note also that ${\rm ks}(V)=0$ if and only if $m+i\equiv 0$ mod $2$.

As examples of $M$, we can take $M=T^2\times S^2$ or $T^4$ or their arbitrary connected sum.
In fact, $b_2(M;l^\prime)=0$ and $w_1(\lambda^\prime)^2=0$ for any nontrivial $\Z$-bundle $l^\prime$  over $M=T^2\times S^2$ or $T^4$.
When $M$ is a connected sum of several $T^2\times S^2$ or $T^4$, take $l^\prime$ which is nontrivial on each $T^2\times S^2$ or $T^4$ summand.
\begin{proof}[Proof of \thmref{thm:nonsmooth}]
Suppose $V$ satisfies (1) and $X$ is smoothable.
Take $l^\prime$ as in the assumption, and let $l\to V\#M$ be the connected sum of a trivial $\Z$-bundle on $V$ and $l^\prime$.
Then, $H^2(X;l)=H^2(V;\Z)\oplus H^2(M;l^\prime)$ and $Q_{X,l} = Q_V$.
Note that $w_1(\lambda)^2 =w_1(\lambda^\prime)^2=0$.
By \thmref{thm:main}, $Q_{X,l}$ should be standard. 
This is a contradiction.
If $V$ satisfies (2), then consider $l=l_\alpha\#l^\prime$ and use \thmref{thm:10/8}.
%\qed
\end{proof}
\begin{proof}[Proof of \thmref{thm:spin}]
Let $V$ be any simply-connected $4$-manifold with even form $Q_V$ of rank $16k$ which satisfies either of the following:
\begin{enumerate}
\item $Q_V$ is definite, or
\item $Q_V\cong m(-E_8)\oplus n H$ and $m>n$, where $H$ is the hyperbolic form. 
\end{enumerate}
Then, take a connected sum of $V$ with sufficiently many $T^2\times S^2$'s or $T^4$'s so that the $10/8$-inequality is satisfied.
By \thmref{thm:nonsmooth}, it is nonsmoothable.
%\qed
\end{proof}
\begin{proof}[Proof of \thmref{thm:non-spin}]
Let $V=m|E_8|\# n(S^2\times S^2)\#\Sigma_i$ with $m>n+1$, and take a connected sum of $V$ with sufficiently many $T^2\times S^2$'s or $T^4$'s.
%\qed
\end{proof}
%
%

%
%
%%%%%%%%%%%%%%%%%%%%%%%%%%%%%%%%%%%%%%%%%%%%%%%%%%%%%%%%%%%%%%%%%%%%%%%%%%%%%%%
\section{$\Spin^{c_-}$-structures}\label{sec:spin-p}
%%%%%%%%%%%%%%%%%%%%%%%%%%%%%%%%%%%%%%%%%%%%%%%%%%%%%%%%%%%%%%%%%%%%%%%%%%%%%%%
%
%

In this section, we introduce a variant of $\Spin^c$-structure, {\it $\Spin^{c_-}$-structure} we call.
The notion of {\it $\Spin^{c_-}$-structure} was introduced to the author by M.~Furuta, and a large part of this section is due to him.

%%%%%%%%%%%%%%%%%%%%%%%%%%%%%%%%
\subsection{$\Spin^{c_-}$-groups}
%%%%%%%%%%%%%%%%%%%%%%%%%%%%%%%%

Let $\Pin^-(2)$ be the subgroup of $\SP(1)$ generated by 
$\U(1)$ and $j$, that is, $\Pin^-(2) = \U(1)\cup j\U(1)$.
There is a two-to-one homomorphism $\varphi_0\colon\Pin^-(2)\to \OO(2)$, which sends $z\in\U(1)$ in $\Pin^-(2)$ to $z^2\in \U(1)\subset \OO(2)$, and $j$ to the reflection
$$
\begin{pmatrix}
1 & 0\\
0 & -1
\end{pmatrix}.
$$

Let us define $\Spin^{c_-}(n) = \Spin(n)\times_{\{\pm 1\}}\Pin^-(2)$. 
There is an exact sequence
$$
1\to\{\pm 1\}\to \Spin^{c_-}(n) \to \SO(n)\times\OO(2) \to 1.
$$

%%%%%%%%%%%%%%%%%%%%%%%%%%%%%%%%
\subsection{$\Spin^{c_-}$-structures}
%%%%%%%%%%%%%%%%%%%%%%%%%%%%%%%%

Let $X$ be a $n$-dimensional oriented smooth manifold.
Fix a Riemannian metric on $X$,  and let $F(X)$ be its $\SO(n)$-frame bundle.
Suppose an $\OO(2)$-bundle $E$ over $X$ is given.
\begin{Definition}
A $\Spin^{c_-}$-structure on $(X,E)$ is a lift of the principal $\SO(n)\times\OO(2)$-bundle $F(X)\times_X E$ to a principal $\Spin^{c_-}(n)$-bundle.
This is given by the data $(P,\tau)$ where $P$ is a $\Spin^{c_-}(n)$-bundle and $\tau$ is a bundle isomorphism $P/\{\pm 1\}\to F(X)\times_X E$.
\end{Definition}
\begin{Remark}
More generally, one can define a $\Spin^{c_-}$-structure on the pair $(V,E)$ of an $\SO(n)$-bundle $V$ and an $\OO(2)$-bundle $E$ over $X$ as a $\Spin^{c_-}(n)$-lift of $V\times_X E$.
\end{Remark}
\begin{Remark}
Let $G_0$ be the identity component of $\Spin^{c_-}(n)$. 
Then $G_0$ is isomorphic to $\Spin^c(n)$, and $\X=P/G_0\to X$ is a double covering.
Note that the determinant line bundle $\det E$ of $E$ is isomorphic to $\X\times_{\{\pm 1\}}\R$, where $\{\pm 1\}$ acts on $\R$ by multiplication.
\end{Remark}
\begin{Proposition}\label{prop:spin-p}
There exists a  $\Spin^{c_-}$-structure on $F(X)\times_X E$ if and only if $w_2(TX)=w_2(E)+ w_1(E)^2$.
\end{Proposition}
\begin{proof}
Note that the image of $\Pin^-(2)\subset \SP(1)=\Spin(3)$ by the canonical homomorphism $\Spin(3)\to \SO(3)$ is a copy of  $\OO(2)$ embedded in $\SO(3)$.
This embedding $\OO(2)\subset\SO(3)$ is given by $A\mapsto A\oplus\det A$.
By using this embedding, embed $\SO(n)\times\OO(2)$ in $\SO(n+3)$.
Then we have a commutative diagram
$$
\begin{CD}
1@>>>\{\pm 1\}@>>>\Spin^{c_-}(n)@>>>\SO(n)\times\OO(2)@>>>1\\
@.@|@VVV @VVV @. @.\\
1@>>>\{\pm 1\}@>>>\Spin(n+3)@>>>\SO(n+3)@>>>1.
\end{CD}
$$
The diagram leads to a commutative diagram of fibrations
$$
\begin{CD}
K(\Z_2,1)@>>>B\Spin^{c_-}(n)@>>>B\SO(n)\times B\OO(2)@>>>K(\Z_2, 2)\\
@VVV @VVV @VVV @VVV\\
K(\Z_2,1)@>>>B\Spin(n+3)@>>>B\SO(n+3)@>{w_2}>>K(\Z_2, 2).
\end{CD}
$$
From these, we see that 
$$w_2(TX\oplus E\oplus\det E)=w_2(X)+ w_2(E)+w_1(E)^2=0$$  
is the required condition.
%\qed
\end{proof}
\begin{Remark}\label{rem:E}
Let $l\to X$ be a $\Z$-bundle over $X$, and $\lambda = l\otimes\R$.
The isomorphism classes of $\OO(2)$-bundles $E$ whose determinant line bundles $\det E$ are isomorphic to $\lambda$ are classified by their twisted first Chern classes $\tilde{c}_1(E)\in H^2(X;l)$. See \cite{Froyshov}, Proposition 2.2.
Note also that $\tilde{c}_1(E)=0$ if and only if $E$ is isomorphic to $\underline{\R}\oplus \lambda$, where $\underline{\R}$ is a trivial $\R$-bundle over $X$.
\end{Remark}
We concentrate on the case when $n=4$ below.
Let $\HH_T$ be a $\Spin^{c_-}(4)$-module which is a copy of $\HH$ as a vector space, such that the action of  $[q_+,q_-,u]\in \Spin^{c_-}(4)=(\SP(1)\times\SP(1))\times_{\{\pm 1\}}\Pin(2)$ on $v\in \HH_T$ is given by $q_+vq_-^{-1}$. 
Then, the associated bundle $P\times_{\Spin^{c_-}(4)}\HH_T$ is identified with the tangent bundle $TX$.

Similarly, let $\bar{\varphi}\colon \Spin^{c_-}(4)\to \OO(2)$ be the homomorphism defined from $\varphi_0\colon \Pin^-(2)\to \OO(2)$.
Then the associated bundle $P\times_{\bar{\varphi}}\OO(2)$ is identified with $E$.

Let us consider $\Spin^{c_-}(4)$-modules $\HH_+$ and $\HH_-$ which are copies of $\HH$ as vector spaces, such that the action of $[q_+,q_-,u]\in \Spin^{c_-}(4)$ on $\phi\in \HH_{\pm}$ is given by $q_{\pm}\phi u^{-1}$.
Then, one can obtain the associated bundles $S^\pm = P\times_{\Spin^{c_-}(4)}\HH_\pm$. 
These are {\it positive} and {\it negative} {\it spinor bundles} for the $\Spin^{c_-}$-structure.

The Clifford multiplication $\rho_\R\colon \Omega^1(X)\times \Gamma(S^+)\to \Gamma(S^-)$ is defined via $\Spin^{c_-}(4)$-equivariant map 
$\HH_T\times \HH_+\to \HH_-$ defined by $(v,\phi) \mapsto \bar{v}\phi$.
Later we will need a {\it twisted complex} version of the Clifford multiplication defined as follows.
Let $G_0$ be the identity component of $\Spin^{c_-}(4)$. Then $G_0$ is isomorphic to $\Spin^c(4)$, and $\Spin^{c_-}(4)/G_0\cong \{\pm 1\}$.
Let $\varepsilon\colon \Spin^{c_-}(4)\to \Spin^{c_-}(4)/G_0$ be the projection, and let $\Spin^{c_-}(4)/G_0\cong \{\pm 1\}$ act on $\C$ by complex conjugation.
Then $\Spin^{c_-}(4)$ acts on $\C$ via $\varepsilon$ and complex conjugation. 
Define 
$$
\rho_0\colon \HH_T\otimes_{\R}\C \times \HH_+\to \HH_-
$$
 by $\rho_0(v\otimes a,\phi)= \bar{v}\phi\bar{a}$.
This $\rho_0$ is $\Spin^{c_-}(4)$-equivariant.
Let us define the bundle $K$ over $X$ by $K=\X\times_{\{\pm 1\}}\C$ where $\{\pm 1\}$ acts on $\C$ by complex conjugation.
Then we can define  via $\rho_0$ the Clifford multiplication 
\begin{equation}\label{eq:clif}
\rho\colon \Omega^1(X;K)\times \Gamma (S^+)\to\Gamma (S^-).
\end{equation}
Note that $K=\underline{\R}\oplus i\lambda$, where $\underline{\R}$ is a trivial $\R$-bundle.
By restricting  $\rho$ to $\underline{\R}$, $\rho_\R$ is recovered. 
By restricting $\rho$ to  $i\lambda$, we obtain
\begin{equation*}%\label{eq:clif}
\rho\colon \Omega^1(X;i\lambda)\times \Gamma (S^+)\to\Gamma (S^-).
\end{equation*}

%%%%%%%%%%%%%%%%%%%%%%%
\subsection{The relation with $\Spin^c$-structures on the double covering}\label{subsec:covering}
%%%%%%%%%%%%%%%%%%%%%%

In this subsection, we write  $\Spin^{c_-}(4)$ as $G$. 
Note that $G$ has two connected components $G_0$ and $G_1$, and the identity component $G_0$ is $\Spin^c(4)$.
If a $\Spin^{c_-}$-structure $(P,\tau)$ on a $4$-manifold $X$ is given, then $\tilde{X}=P/G_0$ gives a double covering $\pi\colon\tilde{X}\to X$.
Then, we have a $G_0$-bundle $P\to P/G_0=\tilde{X}$.
The pull-back bundle $\pi^*E$ has an $\SO(2)$-reduction $L$, and a bundle isomorphism $\tilde{\tau}\colon P/\{\pm 1\}\to F(\X) \times_{\X}L$ is induced from $\tau$, where $F(\X) =\pi^*F(X)$, which can be considered as the frame bundle over $\X$ for the pull-back metric.
The $G_0$-bundle $P$ over $\X$ and $\tilde{\tau}$ define an ordinary $\Spin^c$-structure $\tilde{c}$ on $\tilde{X}$.

Let $\iota\colon\X\to\X$ be the covering transformation, and define $J$ by 
$$
J=[1,j^{-1}]\in G_1=\Spin(4)\times_{\{\pm 1\}}j\U(1). %G=\Spin(4)\times_{\{\pm 1\}}\Pin^-(2).
$$
Then the right $J$-action on $P\to\X$ covers the $\iota$-action.
Although the $J$-action is not a $G_0$-bundle automorphism of $P\to\X$, 
it can be considered as the composition of the following two maps of $G_0$-bundles:
\begin{itemize}
\item A $G_0$-bundle map covering the $\iota$-action, $\tilde{\iota}\colon P \to \bar{P}$, where $\bar{P}$ is the $G_0$-bundle for the complex conjugate $\Spinc$-structure of $\tilde{c}$. 
\item The complex conjugation, $\alpha\colon\bar{P}\to P$, covering the identity map of $\X$.
\end{itemize}

To see this, let us consider the pull-back $G$-bundle $\pi^*P\to \X$.
Then 
$$
\pi^*P = P\times_G(G/G_0\times G) = P\times_G(\{\pm 1\}\times G) = P\times_{G_0}G.
$$
The bundle $P\times_{G_0}G$ has two components: $P\times_{G_0}G =P_0\sqcup P_1$, where $P_i = P\times_{G_0}G_i$ for $i=0,1$.
Since the right $G_0$-action on $\pi^*P$ preserves $P_0$ and $P_1$,  $P_0$ and $P_1$ are considered as $G_0$-bundles over $\X$ by this $G_0$-action.  
Note that $P_0 = P\times_{G_0}G_0$ is isomorphic to $P$ as $G_0$-bundles over $\X$.
The isomorphism $\alpha_0\colon P\times_{G_0}G_0 \to P$ is given by $\alpha_0([p,g])= pg$ for $[p,g]\in P\times_{G_0}G_0$.

On the other hand, $P_1$ can be identified with the complex conjugation $\bar{P}$ of $P$ as follows:
For $g= [s,u]\in \Spin(4)\times_{\{\pm 1\}} \U(1)=G_0$, let $\bar{g}$ be $[s,u^{-1}]$. 
Every element $g^\prime\in  \Spin(4)\times_{\{\pm 1\}} j\U(1)=G_1$ can be written as $g^\prime=J^{-1}g=\bar{g}J^{-1}$ for some $g\in G_0$.
Let us define the fiber-preserving diffeomorphism $\alpha_1\colon P_1=P\times_{G_0}G_1\to P$ by $\alpha_1([p,\bar{g}J^{-1}]) = p\bar{g}$.
Then, for $q\in P_1$ and $h\in G_0$, 
$$
\alpha_1(qh) = \alpha_1(q)\bar{h}.
$$
This means $P_1\cong \bar{P}$ as $G_0$-bundles.
Let us define $\alpha\colon P_1\to P_0$ by $\alpha=\alpha_0^{-1}\circ\alpha_1$. 
Explicitly,  $\alpha([p,g^\prime]) = [p,g^\prime J]$.

The map $\iota\colon\X\to\X$ has a natural lift $\tilde{\iota}\colon P\times_{G_0}G\to P\times_{G_0}G$ given by $\tilde{\iota}([p,g]) = [pJ,J^{-1}g]$.
Note that $\tilde{\iota}$ exchanges the components $P_0$ and $P_1$.
Then the $J$-action on $P$ can be identified with the composition $\alpha\circ\tilde{\iota}\colon P_0 \to P_0$.

The $J$-action also induces antilinear automorphisms, denoted by $I$, on the spinor bundles $\tilde{S}^\pm = P\times_{G_0}\HH_\pm$ given by $I([p,\phi])=[pJ,J^{-1}\cdot\phi ] = [pJ, \phi j^{-1}]$. 
(In the expression $J^{-1}\cdot\phi $, "$\cdot$" means the $G$-action on $\HH^\pm$.)
Since $J^2\in G_0$, $I^2 ([p,\phi])=[pJ^2,J^{-2}\cdot\phi]=[p,\phi]$.
Therefore  $I$ is an antilinear involution on each of spinor bundles. 
The relation between the $\Spin^c$-spinor bundles $\tilde{S}^\pm$ over $\X$ and the $\Spin^{c_-}$-spinor bundles $S^\pm$ over $X$ is given by
\begin{equation}\label{eq:spinor}
\tilde{S}^\pm \cong \pi^*S^\pm, \quad S^\pm \cong \tilde{S}^\pm/I.
\end{equation}

Similarly, the $J$-action induces an antilinear involution of the determinant line bundle, also denoted by $I$.
This can be seen from the construction above, or noticing the following.
Note that $\lambda = \X\times_{\{\pm 1\}}\R\to X$ is isomorphic to the determinant $\R$-bundle of $E$.
Let $E_0\to X$ be the $\R^2$-bundle associated to $E$. 
Then, the determinant $\C$-bundle $L_0$ of $\tilde{c}$ can be identified with the pull-back $\pi^*E_0$ as {\it real} vector bundles, and the involution $\iota$ lifts to $L_0\cong \pi^*E_0$ as an involutive antilinear bundle automorphism.

\begin{Remark}
By using $\Pin^+(2)$ (instead of $\Pin^-(2)$), we can define analogous objects, {\it $\Spin^{c_+}$-structures}.
A definition of $\Pin^+(2)$ is given as follows: 
Let us consider the embedding of $\OO(2)$ into $\SO(5)$ defined by 
$$\OO(2)\ni A\mapsto A\oplus\det A\oplus\det A\oplus \det A \in\SO(5),$$
 and let $\varphi\colon\Spin(5)\to\SO(5)$ be the canonical homomorphism. 
Then $\Pin^+(2)$ is defined by $\Pin^+(2)=\varphi^{-1}(\OO(2))$. 
It can be seen that $\Pin^+(2)$  is isomorphic to $\OO(2)$ which is  considered as a double covering of $\OO(2)$.
(On the other hand, $\Pin^-(2)$ can be defined via the embedding of $\OO(2)$ into $\SO(3)$ defined by $\OO(2)\ni A\mapsto A\oplus\det A \in\SO(3)$.)

In the case of $\Spin^{c_+}$-structures also, one can construct a $\Spin^c$-structure $\tilde{c}$ associated to it on a double covering $\tilde{X}$ of $X$.  
But the covering transformation $\iota$ lifts on the spinor bundles  as a $\Z/4$-action. 
\end{Remark}

%
%
%%%%%%%%%%%%%%%%%%%%%%%%%%%%%%%%%%%%%%%%%%%%%%%%%%%%%%%%%%%%%%%%%%%%%%%%%%%%%%%
\section{$\Pin^-(2)$-monopole equations}\label{sec:Pin2-monopole}
%%%%%%%%%%%%%%%%%%%%%%%%%%%%%%%%%%%%%%%%%%%%%%%%%%%%%%%%%%%%%%%%%%%%%%%%%%%%%%%
%
%
In this section, we introduce $\Pin^-(2)$-monopole equations, and develop the  $\Pin^-(2)$-monopole gauge theory. 
The whole story is almost parallel to the ordinary Seiberg-Witten case.
%
%
%%%%%%%%%%%%%%%%%%%%%%%
\subsection{Dirac operators}\label{subsec:Dirac}
%%%%%%%%%%%%%%%%%%%%%%

Let $X$ be a closed connected oriented smooth $4$-manifold, $E$ be a $\OO(2)$-bundle over $X$, and $\lambda =\det E$. 
We suppose  $\lambda$ is a nontrivial bundle throughout the rest of the paper.
Fix a Riemannian metric on $X$.
Suppose a $\Spin^{c_-}$-structure $(P,\tau)$ on $(X,E)$ is given.
If an $\OO(2)$-connection $A$ on $E$ is given, then $A$ and the Levi-Civita connection induces a $\Spin^{c_-}(4)$-connection on $P$, and we can define the Dirac operator via the Clifford multiplication $\rho$ of \eqref{eq:clif} as
$$
D_A\colon \Gamma (S^+)\to \Gamma (S^-).
$$
The Dirac operator $D_A$ also have properties similar to the ordinary Dirac operators.
If $A^\prime$ is another $\OO(2)$-connection on $E$, then $a = A-A^\prime$ is in $\Omega^1(X;i\lambda)$, and the relation of Dirac operators of $A$ and $A^\prime=A+a$ is given via $\rho$ by
$$
D_{A+a} \Phi = D_A\Phi + \frac12\rho(a)\Phi. 
$$
While the ordinary spinor bundles are equipped with the canonical hermitian inner products, the spinor bundles for a $\Spin^{c_-}$-structure do not have such hermitian inner products.
However, the pointwise {\it twisted} hermitian inner product 
\begin{equation}\label{eq:tw-herm}
\langle\cdot,\cdot\rangle_{K,x} \colon S^\pm_x\times S^\pm_x \to K_x  
\end{equation}
is naturally  defined, where the objects with the subscription $x$ means the fibers over $x\in X$, and $K=\X\times_{\{\pm 1\}}\C$. 
The precise meaning is as follows:
Let $\tilde{S}^\pm$ be the spinor bundles of the associated $\Spin^c$-structure on the double covering $\X$. 
Then the canonical hermitian inner product of  $\tilde{S}^\pm$ can be given as the bundle homomorphisms
\begin{equation}\label{eq:herm}
\tilde{S}^\pm\otimes \tilde{S}^\pm \to \underline{\C},
\end{equation}
where $\underline{\C}$ is a trivial bundle $\X\times\C$. 
The diagonal action of $I$ on $\tilde{S}^\pm\otimes \tilde{S}^\pm$ is an involution, also denoted by $I$.
Let us define the involution $I$ on  $\underline{\C}=\X\times\C$ by $I(x,v)=(\iota x,\bar{v})$, where $\bar{v}$ is the complex conjugation of $v$.
Then $\eqref{eq:herm}$ is $I$-equivariant.
Dividing \eqref{eq:herm} by $I$, we obtain the bundle homomorphism
$$
S^\pm\otimes S^\pm\to K,
$$
which gives the twisted hermitian inner product \eqref{eq:tw-herm}.

The real part of \eqref{eq:tw-herm}
$$
\langle\cdot,\cdot\rangle_{\R,x} =\real\langle\cdot,\cdot\rangle_{K,x} 
$$
defines a real inner product on $S^\pm$.
Then it is easy to see that the Dirac operator is formally self-adjoint with respect to the $L^2$-inner product induced from $\langle\cdot,\cdot\rangle_{\R} $.
({\it Cf.} \cite{Morgan}, Lemma 3.3.3.)
\begin{Proposition}
Suppose $X$ is closed and a $\Spin^{c_-}$-structure on $X$ is given. 
Then its Dirac operator is formally self-adjoint in the sense that
$$
(\Phi,D_A\Psi)_{L^2} = (D_A\Phi,\Psi)_{L^2},
$$
where 
$$
(\Phi_1,\Phi_2)_{L^2} = \int_X \langle\Phi_1,\Phi_2\rangle_\R dvol.
$$
\end{Proposition}
%
%

%%%%%%%%%%%%%%%%%%%%%%%
\subsection{$\Pin^-(2)$-monopole equations}
%%%%%%%%%%%%%%%%%%%%%%
The curvature $F_A$ of $A$ is an element of $\Omega^2(X;i\lambda)$.
The space of $i\lambda$-valued self-dual forms, $\Omega^+(X;i\lambda)$, is also associated to the $\Spin^{c_-}(4)$-bundle $P$ as follows.
Let $\varepsilon\colon \Pin^-(2)\to\Pin^-(2)/\U(1)\cong \{\pm 1\}$ be the projection, and let $\Spin^{c_-}(4)$ act on $\im\HH$ by $v\in \im\HH \to \varepsilon(u)q_+vq_+^{-1}$ for $[q_+,q_-,u]\in \Spin^{c_-}(4)$. 
Then the space of sections of the associated bundle $P\times_{\Spin^{c_-}(4)}\im\HH$ is isomorphic to $\Omega^+(X;i\lambda)$.
For $\phi\in \HH_+$, $\phi i \bar\phi\in\im \HH$, and  $\Spin^{c_-}(4)$ acts on it similarly.
Thus, one can define a quadratic map 
$$
q\colon \Gamma(S^+)\to\Omega^+(X;i\lambda).
$$

Let $\A(E)$ be the space of  $\OO(2)$-connections on $E$. 
Then $\Pin^-(2)$-monopole equations for  $(A,\Phi)\in \A(E)\times \Gamma(S^+)$ are defined by 
\begin{equation}\label{eq:monopole}
\left\{
\begin{aligned}
D_A\Phi =& 0,\\
F_A^+ =& q(\Phi),
\end{aligned}\right.
\end{equation}
where $F_A^+$ is the self-dual part of the curvature $F_A$.

As in the case of the ordinary Seiberg-Witten equations, it is convenient to  work in Sobolev spaces.
Fix $k\geq 4$, and take $L^2_k$-completion of $\A(E)\times \Gamma(S^+)$.
The $\Pin^-(2)$-monopole equations \eqref{eq:monopole} are assumed as equations for  $L^2_k$-connections/spinors.

%%%%%%%%%%%%%%%%%%%%%%%
\subsection{Gauge transformations}\label{subsec:G}
%%%%%%%%%%%%%%%%%%%%%%
The gauge transformation group $\G$ is defined as the space of $\Spin^{c_-}(4)$-equivariant diffeomorphisms of $P$ covering the  identity map of the quotient $P/\Pin^-(2)$.
Then, $\G$ can be identified with $\Gamma(P\times_{\ad}\Pin^-(2))$, where $\ad$ means the adjoint representation on $\Pin^-(2)$ by the $\Pin^-(2)$-component of $\Spin^{c_-}(4)$.
Note that $\Lie\G \cong\Gamma(P\times_{\ad}i\R )\cong \Omega^0(X;i\lambda)$.
We take $L^2_{k+1}$-completion of $\G$.

Let us look at $\G$ more closely.
Recall that $\Pin^-(2) =\U(1)\cup j\U(1)$.
For $u,z\in\U(1)$, note that 
\begin{equation}\label{eq:ad}
\begin{aligned}
\ad_z(u) &=zuz^{-1} =u,\\ 
\ad_{jz}(u) &= jzuz^{-1}j^{-1}=u^{-1},\\ 
\ad_z(ju)&= z^2ju =z^2u^{-1}j, \\
\ad_{jz}(ju)&=z^{-2}ju^{-1}=z^{-2}uj.
\end{aligned}
\end{equation}
Therefore the adjoint action preserves the component of $\Pin^-(2)$.
Then $\G$ can be decomposed into $\G=\G_0\cup\G_1$, where $\G_0=\Gamma(P\times_{\ad}\U(1))$ and  $\G_1=\Gamma(P\times_{\ad}j\U(1))$.
To understand $\G_0$ and $\G_1$, we note the next proposition.
\begin{Proposition}\label{prop:ad}
The bundle $P\times_{\ad}\U(1)$ is identified with $\X\times_{\{\pm 1\}}\U(1)$, where $\{\pm 1\}$ acts on $\U(1)$ by complex conjugation.
The bundle $P\times_{\ad}j\U(1)$ is identified with the bundle $S(E)$ of unit vectors of $E$.
\end{Proposition}
\begin{proof}
By \eqref{eq:ad}, the adjoint action of $G=\Spin^{c_-}(4)$ on $\U(1)$ is given by complex conjugation via the projection $G\to G/G_0=\{\pm 1\}$.  
Therefore,
$$
P\times_{\ad} \U(1) =P/G_0\times_{\{\pm 1\}}\U(1)=\X\times_{\{\pm 1\}}\U(1).
$$
Let $G$ act on $\C$ as follows: 
For $g=[s,u]\in \Spin(4)\times_{\{\pm 1\}} \U(1)$ and $w\in \C$, define the $g$-action on $w$ by $g\cdot w=z^2w$.  
For $J^\prime=[1,j]\in \Spin(4)\times_{\{\pm 1\}} j\U(1)$ and $w\in \C$, define the $J^\prime$-action on $w$ by $J^\prime\cdot w=\bar{w}$.
Then the associated bundle $P\times_G\C$ is isomorphic to $E$.
Let us embed $j\U(1)$ into $\C$ by
$$
j\U(1)\ni ju=u^{-1}j \mapsto u^{-1}\in \U(1)\subset\C.
$$
By \eqref{eq:ad}, this gives the identification between $P\times_{\ad}j\U(1)$ and $S(E)$.
%\qed
\end{proof}
In fact, $\G_1$ is empty except one case.
\begin{Proposition}
$\G_1=\emptyset$ if and only if $\tilde{c}_1(E)\neq 0$.
\end{Proposition}
\begin{proof}
By \propref{prop:ad}, $\G_1\cong \Gamma(S(E))$,  and $\tilde{c}_1(E) =0$ if and only if $E$ is isomorphic to $\underline{\R}\oplus i\lambda$ as $\OO(2)$-bundles with determinant line bundle $\lambda$. (Recall \remref{rem:E}.)
%\qed
\end{proof}

The  $\G$-action on $\A(E)\times\Gamma(S^+)$ is given by $g(A,\Phi)= (A-2g^{-1}dg, g\Phi)$, for $g\in\G$ and $(A,\Phi)\in \A(E)\times\Gamma(S^+)$.
If $\Phi\not\equiv 0$, then $\G$-action on  $(A,\Phi)$ is free, and such an $(A,\Phi)$ is called an {\it irreducible}.
On the other hand, $(A,\Phi)$ with $\Phi\equiv 0$ is called a {\it reducible}.
The stabilizer of the $\G$-action on $(A,0)$ is the subgroup of constant sections $\{\pm 1\}\subset \G_0$, unless 
$E\cong\underline{\R}\oplus\lambda$ and $A$ is flat.
If $E\cong\underline{\R}\oplus\lambda$ and $A$ is flat, then the stabilizer is generated by the constant section $j\in\G_1$, and is isomorphic to $\Z/4$.

%
%
%%%%%%%%%%%%%%%%%%%%%%%
\subsection{Moduli spaces}
%%%%%%%%%%%%%%%%%%%%%%
Let us define the moduli spaces $\M$ and $\M_0$ of $\Pin^-(2)$-monopoles as follows: 
$$
\M=\{\text{ solutions to \eqref{eq:monopole} }\}/\G,\quad \M_0 =\{\text{ solutions to \eqref{eq:monopole} }\}/\G_0.
$$
Then, $\M_0=\M$ unless $\tilde{c}_1(E)=0$.
If $\tilde{c}_1(E)=0$, then $\M_0$ is a double covering of $\M$.
\begin{Proposition}\label{prop:compact}
The moduli spaces $\M$ and $\M_0$ are compact.
\end{Proposition}
For the Dirac operators of $\Spin^{c_-}$-structures, one can readily prove the Weitzenb\"{o}ck formula (see \cite{Morgan}, Proposition 5.1.5), 
\begin{equation}\label{eq:Weitz}
D_A^2\phi =\nabla^*_A\nabla_A\phi + \frac{\kappa}4\phi + \frac{\rho(F_A)}2\phi,
\end{equation}
where $\kappa$ is the scalar curvature of the metric on $X$.
With this understood, the proof of \propref{prop:compact} is parallel to the case of the ordinary Seiberg-Witten theory.
The compactness of $\M$ can be seen also from the relation with the Seiberg-Witten theory on the double covering as in the next subsection.

%%%%%%%%%%%%%%%%%%%%%%%
\subsection{The relation with the Seiberg-Witten theory on the double covering}
%%%%%%%%%%%%%%%%%%%%%%
Let $\A(E)$ be the space of $\OO(2)$-connections on $E$.
As explained in \secref{subsec:covering}, for a $\Spin^{c_-}$-structure on $(X,E)$, it is induced a $\Spin^c$-structure $\ct$ on the double covering $\X$ associated to $\lambda=\det E$.
Let $\pi\colon \X\to X$ be the projection and $\iota\colon\X\to\X$ be the covering transformation.
Let $\tilde{S}^\pm$ be the spinor bundles of $\ct$, $L$ be the determinant line bundle of $\ct$, and $\A(L)$ be the space of $\U(1)$-connections on $L$.
In this situation, the $I$-action on $\CC:=\A\times \Gamma(\tilde{S}^+)$ is induced from the $I$-action on $\tilde{S}^\pm$ and $L$. 
Then, by \eqref{eq:spinor},  
$$
\Gamma(S^\pm)\cong \Gamma(\tilde{S}^\pm)^I.
$$
The relation of $\A(E)$ and $\A(L)$ is given as follows.
An $\OO(2)$-connection $A$ on $E$ and the Levi-Civita connection determine a $\Spin^{c_-}(4)$-connection $\Bbb A$ on $P$.  
Let us consider the pull-back $\Spin^{c_-}(4)$-connection $\pi^* {\Bbb A}$ on $\pi^* P\to \X$. 
Since $\pi^*P = P_0\cup P_1$ (see \subsecref{subsec:covering}), the $\Spin^{c_-}(4)$-connection $\pi^* {\Bbb A}$ has a $\Spin^c(4)$-reduction $\tilde{\Bbb A}$ on the $\Spin^c(4)$-bundle $P_0$.
Then we obtain a $\U(1)$-connection $\tilde{A}$ on $L$ from $\tilde{\Bbb A}$, and we can see that
$$
\A(E)\cong \A(L)^I.
$$

The gauge transformation group on $\X$ is given by $\tilde{\G}=\Map(\X,S^1)$.
If we define the involution $I$ on $\tilde{\G}$ by $I\tilde{u}=\overline{\iota^*\tilde{u}}$ for $\tilde{u}\in\tilde{\G}$, then the $\G$-action on $\A(L)\times\Gamma(\tilde{S^+})$ is $I$-equivariant, and  $\tilde{\G}^I\cong \Gamma(\X\times_{\{\pm 1\}}\U(1))\cong \G_0$.
Let $\CC=\A(E)\times\Gamma(S^+)$ and $\tilde{\CC} = \A(L)\times \Gamma(\tilde{S}^+)$. 
Then we have
\begin{Proposition}
$\CC/\G_0 \cong \tilde{\CC}^I/\tilde{\G}^I$. 
\end{Proposition}

Via the identifications above, the $\Spin^{c_-}$-Dirac operator $D_A\colon \Gamma(S^+)\to\Gamma(S^-)$ can be identified with the restriction of the $\Spin^c$-Dirac operator $D_{\tilde{A}}$ on $(\X,\ct)$ to the $I$-invariant part, $D_{\tilde{A}}\colon \Gamma(\tilde{S}^+)^I\to \Gamma(\tilde{S}^-)^I$.
Furthermore, the Seiberg-Witten equations on $(\X,\ct)$ is $I$-equivariant in our setting. 
Let us define the $I$-invariant Seiberg-Witten moduli space $\tilde{\M}^I$  as the space of $I$-invariant solutions divided by $\tilde{\G}^I$. 
Then, it can be checked that  $\tilde{\M}^I$  can be identified with the $\Pin^-(2)$-monopole moduli space $\M_0$:
\begin{Proposition}\label{prop:MM}
$\M_0\cong \tilde{\M}^I$.
\end{Proposition}
\begin{Remark}
The $I$-invariant moduli $\tilde{\M}^I$ can be embedded in the ordinary Seiberg-Witten moduli space $\tilde{\M}$ of $(\X,\ct)$, since $\tilde{\CC}^I/\tilde{\G}^I$ is continuously embedded in $\tilde{\CC}/\tilde{\G}$ ({\it cf.} Remark 3.4 of \cite{Nfree} or \cite{FurutaG}).
Since $\tilde{\M}$ is compact, the compactness of $\M_0$ (\propref{prop:compact}) follows from \propref{prop:MM}, too. 
\end{Remark}
\begin{Remark}
In general, $\M$ ($\tilde{\M}^I$) could be non-orientable.
A similar but slightly different situation is studied by Tian-Wang \cite{TW}.
They investigate the Seiberg-Witten theory in the presence of real structures on almost complex $4$-manifolds.
They also introduce an antilinear involution on the Seiberg-Witten theory. 
Their involution  is different from ours in that they use the real structure to define the involution.
\end{Remark}
\begin{Remark}\label{rem:unique}
Since the $\Spin^{c_-}$-Dirac operator $D_A$ is the $I$-invariant part of 
$D_{\tilde{A}}$, the unique continuation theorem holds also for $D_A$. 
Of course, this can be proved directly.
\end{Remark}
%
%%%%%%%%%%%%%%%%%%%%%%%
\subsection{The deformation complex}
%%%%%%%%%%%%%%%%%%%%%%

When $(A,\Phi)$ is a solution of $\Pin^-(2)$-monopole equation, the deformation complex for $\M$ $(\M_0)$ at $(A,\Phi)$ is given as follows:
\begin{equation}\label{eq:deform}
0\to \Omega^0(X;i\lambda)\overset{\alpha}{\to}\Omega^1(X;i\lambda)\oplus\Gamma(S^+) \overset{\beta}{\to} \Omega^+(X;i\lambda)\oplus\Gamma(S^-)\to 0,
\end{equation}
where the maps $\alpha$ and $\beta$ are the linearizations of the $\G$-action and the $\Pin^-(2)$-monopole equations, and given by
$\alpha(f)=(-2df,f\Phi)$, $\beta(a,\phi)= (D_A\phi+\frac12 \rho(a)\Phi, d^+a - Dq_\Phi(\phi))$, where $Dq_\Phi$ is the linearization of $q$ at $\Phi$.

Let $(\tilde{A},\tilde{\Phi})$ be the $I$-invariant solution on $(\X,\ct)$ corresponding to $(A,\Phi)$. Then the deformation complex \eqref{eq:deform} can be identified with the restriction of the ordinary Seiberg-Witten deformation complex at $(\tilde{A},\tilde{\Phi})$ to its $I$-invariant part:
\begin{equation}\label{eq:I-deform}
0\to \Omega^0(\X;i\R)^I\to(\Omega^1(\X;i\R)\oplus\Gamma(\tilde{S}^+))^I \to (\Omega^+(\X;i\R)\oplus\Gamma(\tilde{S}^-))^I\to 0,
\end{equation}
where the $I$-action on forms is given by the composition of the pullback by $\iota$ and the complex conjugation.
For calculation of the index of \eqref{eq:deform}, $0$-th order terms can be neglected, and therefore, the complex \eqref{eq:deform} can be assumed to be a direct sum of the de Rham part and the Dirac part. ({\it Cf.} \cite{Morgan}, 4.6.)
The de Rham part is:
$$
0\to\Omega^0(X;i\lambda)\overset{d}{\to}\Omega^1(X;i\lambda)\overset{d^+}{\to}\Omega^+(X;i\lambda)\to 0.
$$ 
The index of the Dirac part is calculated by applying the Lefschetz formula to the $I$-equivariant Dirac operator $D_{\tilde{A}}$ on $(\X,\ct)$. 
More precisely, since the $I$-action is not complex linear, complexify the operator first, and then apply the Lefschetz formula \cite{AB}. 
Then the index of the Dirac part above is half of the index of $D_{\tilde{A}}$ because the $\iota$-action on $\X$ is free.
Thus we have,
\begin{Proposition}
The virtual dimension $d$ of $\M$ is given by 
\begin{equation}\label{eq:d}
d = \frac14(\tilde{c}_1(E)^2-\sign(X)) - (b_0(X;l) - b_1(X;l) + b_+(X;l)),
\end{equation}
where $\tilde{c}_1(E)\in H^2(X;l)$ is the twisted first Chern class. % in \cite{Froyshov}. 
(See \remref{rem:E}.)
\end{Proposition}
\begin{Remark}
Note that $b_0(X;l) =0$ if $X$ is connected and $l$ is nontrivial.
\end{Remark}

%%%%%%%%%%%%%%%%%%%%%%%
\subsection{The topology of $\CC^*/\G_0$}\label{subsec:B}
%%%%%%%%%%%%%%%%%%%%%%

Let $\CC^*$ be the space of irreducibles, i.e., $\CC^*=\A(E)\times(\Gamma(\tilde{S}^+)\setminus 0)$.
The purpose of this subsection is to prove the following proposition.

\begin{Proposition}\label{prop:B}
The space $\CC^*/\G_0$ has the same homotopy type with $$\RP^\infty\times T^{b_1(X;\lambda)}.$$
\end{Proposition}
The proof is divided into several steps.
\begin{Lemma}\label{lem:contractible}
The space $\CC^*$ is contractible.
\end{Lemma}
\begin{proof}
Note that $\CC^*=\A(E)\times(\Gamma(\tilde{S}^+)\setminus 0)$ is the complement of a linear subspace with infinite codimension. 
Therefore $\CC^*$ has the homotopy type of an infinite dimensional sphere, and is contractible.
%\qed
\end{proof}

Since $\G_0$ acts on $\CC^*$ freely, \lemref{lem:contractible} implies that $\CC^*/\G_0$ has the homotopy type of the classifying space $B\G_0$.
Hence, \propref{prop:B} follows from the next lemma.
\begin{Lemma}\label{lem:GI}
$\G_0\simeq (\Z/2)\times \Z^{b_1(X,l)}$.
\end{Lemma}
\begin{proof}
We will prove $\pi_0\G_0$ is isomorphic to $H^1(X;l)$.
Then the lemma follows because $H^1(X;l)=\Z/2\oplus \Z^{b_1(X,l)}$ which is proved by the universal coefficient theorem.
To prove the isomorphism $\pi_0\G_0\cong H^1(X;l)$, one can use obstruction theory.
As an alternative way of the proof, we use sheaf cohomology.
Let us define the bundles $\lambda$ and $\kappa$ over $X$ by $\lambda=l\otimes\R$ and $\kappa=\X\times_{\{\pm 1\}}S^1$, and 
let $\CC^\infty(\lambda)$ and $\CC^\infty(\kappa)$ be the sheaves on $X$ of germs of $C^\infty$-sections of $\lambda$ and $\kappa$, respectively.
Then there is the short exact sequence of sheaves:
$$
1\to l\to {\cal C}^\infty(\lambda)\to{\cal C}^\infty(\kappa)\to 1.
$$
The long exact sequence is induced:
\begin{multline*}
0\to H^0(X;l)\to H^0(X;\CC^\infty(\lambda))\to H^0(X;\CC^\infty(\kappa))\\
\to H^1(X;l)\to H^1(X;\CC^\infty(\lambda))\to H^1(X;\CC^\infty(\kappa))\to\cdots.
\end{multline*}
Now, the lemma follows because $H^i(X;\CC^\infty(\lambda))=0$ and $\pi_0\G_0\cong H^0(X;\CC^\infty(\kappa))$.
\end{proof}

As mentioned above, the $\G_0$-action on $\A(E)$ is not free.
We will need a subgroup of $\G_0$ which acts on $\A(E)$ freely defined as follows.
Let us take a closed loop $\gamma\colon S^1\to X$ so that the restriction of $\lambda$ to $\gamma$, $\lambda|_\gamma=\gamma^*\lambda$, is a nontrivial $\R$-bundle over $\gamma$.
Let $\tilde\gamma\to\gamma$ be the connected double covering of $\gamma$.
Let us define $\G_\gamma$ by $\G_\gamma= \Gamma(\tilde\gamma\times_{\{\pm 1\}}\U(1))$ where the $\{\pm 1\}$-action on $\U(1)$ is given by complex conjugation.
Then $\G_\gamma$ has the following properties:
\begin{itemize}
\item By restricting $\G_0$ to $\gamma$, we have a surjective homomorphism 
$
\G_0\to \G_\gamma.
$
\item $\G_\gamma$ has two components. 
Therefore $\pi_0\G_\gamma\cong\{\pm 1\}$.
\item By restriction and projection, we have a surjective homomorphism 
$$
\theta_\gamma\colon\G_0\to\pi_0\G_\gamma=\{\pm 1\}.
$$
\end{itemize}
Let us define $\K_{\gamma} =\ker \theta_\gamma$.
\begin{Remark}
Let $\{\pm 1\}$ be the subgroup of constant sections in $\G_0$, and let us consider the exact sequence:
$$
1\to\{\pm 1\}\to\G_0\to \G_0/\{\pm 1\}\to 1.
$$
Then $\G_0/\{\pm 1\}$ is homotopy equivalent to $\Z^{b_1(X;l)}$, and the map $\theta_{\gamma}$ gives a splitting of the sequence.
%Let us recall the sequence \eqref{eq:I-split}.
%If we identify $\{\pm 1\}$ with the subgroup of constant sections in $\G^I$, then we have the following commutative diagram.
%$$
%\begin{CD}
%1@>>>{\cal J}\cap\G^I@>>>\G^I@>{h}>>h(\G^I)@>>> 1\\
%@.@A{i_0}AA@|@A{j}AA\\
%1@>>>\{\pm 1\}@>{i_1}>>\G_0@>>>\G_0/\{\pm 1\}@>>>1,
%\end{CD} 
%$$
%where $i_0$ and $i_1$ are the inclusion maps, and $j$ is well-defined since $\{\pm 1\}\subset \ker h$.
%Both of the horizontal sequences are exact.
%Recall $h(\G^I)$ is isomorphic to $\Z^{b_1(X;l)}$.
%As is proved in the proof of \lemref{lem:GI}, $i_0$ is a homotopy equivalence.
%Hence $j$ is also a homotopy equivalence.
%Then, the map $\theta_{\gamma}$ gives a splitting of the second sequence, and therefore  $\K_{\gamma}$ is homotopic to $\Z^{b_1(X;l)}$.
%The splitting is not canonical and depends on the choice of $\gamma$.
\end{Remark}
%
%

%%%%%%%%%%%%%%%%%%%%%%%%
\subsection{The cut-down moduli space}
%%%%%%%%%%%%%%%%%%%%%%%

Since the moduli space $\M_0$ is not necessarily a manifold, we need to perturb the equations. 
As in the Seiberg-Witten case, we will perturb the second equation of \eqref{eq:monopole} by adding an $i\lambda$-valued self-dual $2$-form. 
On the other hand, as we will see later (\subsecref{subsec:mu}), the whole theory of $\Pin^-(2)$-monopole equations can be considered as a family over a torus $T^{b_1(X;\lambda)}$. 
For our purpose, we will cut down the moduli space along a fiber over a point in $T^{b_1(X;\lambda)}$.
These are the tasks of this subsection.

Let us define the $\G_0$-equivariant map 
$$\muh\colon   L^2_k(\A(E)\times\Gamma(S^+))\times L^2_{k-1}(\Omega^+(i\lambda))\to L^2_{k-1}(\Gamma(S^-)\times\Omega^+(i\lambda))
$$ by 
$$
\muh(A,\Phi,\eta) = (D_A\Phi,F_A^+ -q(\Phi)-\eta),
$$
where $\G_0$ acts on $\Omega^+(i\lambda)$ trivially, and  $\G_0$ is completed by $L^2_{k+1}$.
We suppose $k\geq 4$ so that $L^2_{k-1}\subset C^0$.
(Below we omit the symbol $L^2_m$.)
Let us fix a reference connection $A_0\in\A(E)$.
Then $\A(E)$ can be identified with $A_0+\Omega^1(i\lambda)$.
Let us consider the $L^2$-orthogonal splitting: $\Omega^1(i\lambda)=\ker d\oplus(\ker d)^\perp$.
Then $\muh$ can be considered as a map from $\ker d\times(\ker d)^\perp\times \Gamma(S^+)\times\Omega^+(i\lambda)$.
\begin{Proposition}\label{prop:U}
Suppose $b^+(X;\lambda)=0$.
If $\ind D_{A_0}\geq 0$, then there exists a gauge invariant open-dense subset $\cal U$ of $\ker d\times\Omega^+(i\lambda)$ which has the property that the restriction of $\muh$ to ${\cal U^\prime={\cal U}}\times(\ker d)^\perp\times\Gamma(S^+)$ %, $\muh|_{\cal U^\prime}$, 
has $0$ as regular value. 
\end{Proposition}
To prove \propref{prop:U}, we use the next lemma.
\begin{Lemma}\label{lem:dense}
Suppose $\ind D_{A_0}\geq 0$, and let ${\cal O}$ be the set of $A\in \A(E)$ such that $D_A$ is surjective.
Then $\cal O$ is a gauge invariant open-dense subset of $\A(E)$. 
\end{Lemma}

Although the proof of this lemma is standard, we will give a proof for reader's convenience.
The proof is divided into several steps.
({\it Cf.} \cite{Morgan}, Chapter 6 and \cite{Moore}, \S3.4.)

Let us define $F\colon \A(E)\times \Gamma(S^+)\to \Gamma(S^-)$ by 
$F(A,\Phi)=D_A\Phi$.
Then the differential of $F$ is given by 
$$
DF_{(A,\Phi)}(a,\phi) = D_{A} \phi +\frac12 \rho(a)\Phi, \quad \text{for } (a,\phi)\in \Omega^1(i\lambda)\times \Gamma (S^+).
$$
\begin{Lemma}\label{lem:surj}
If $F(A,\Phi) =0$ and $\Phi\neq 0$, then $DF_{(A,\Phi)}$ is surjective.
\end{Lemma}
\begin{proof}
First, note that, if $\phi_x\in S^+_x$, a spinor vector over $x\in X$, is nonzero, the linear map 
$$
T^*_xX\ni a_x \mapsto \rho(a_x)\phi_x \in S^+_x
$$
is an isomorphism.
Suppose $\psi\in\Gamma(S^-)$ is perpendicular to the image of $DF_{(A,\Phi)}$.
Then by holding $\phi=0$ and varying $a$, we see $\psi=0$ on the support $U$ of $\Phi$ which is open. 
By holding $a=0$ and varying $\phi$, we have $D_A\psi =0$, and by the unique continuation theorem, $\psi$ is identically zero.
Thus the lemma is proved.
%\qed
\end{proof}

By the implicit function theorem,  $N=F^{-1}(0)\cap\{\Phi\neq 0\}$ is a submanifold of $\A(E)\times \Gamma(S^+)$. 
Let $\pi_0\colon N\to \A(E)$ be the projection to the first factor,  $\pi_0(A,\Phi)=A$.
By the standard argument, it is easy to see the following. (See e.g. \cite{Nicolaescu}, \S 1.5.2.)
\begin{Lemma}
The map $\pi_0$ is Fredholm whose index is equal to $\ind D_{A_0}$.
\end{Lemma}

Now let us prove \lemref{lem:dense}.
\begin{proof}[Proof of \lemref{lem:dense}]
Suppose $\ind D_{A_0}$ is nonnegative and let ${\cal O}$ be the set of regular values of $\pi_0$. 
Then $A\in{\cal O}$ if and only if  $D_A$ is surjective.
By the Sard-Smale theorem, $\cal O$ is a dense subspace of $\A(E)$.

The space $\cal O$ is gauge invariant because $F$ and $\pi_0$ is gauge equivariant. 

Let us prove $\cal O$ is open.
Note that $D_A$ is surjective if and only if its formal adjoint $D_A^*\colon L^2_k(S^-)\to L^2_{k-1}(S^+)$ is injective.
That is,
$${\cal O} = \{A\in \A(E)\,|\, \ker D_A^*=0\}.
$$
Let $\bar{\cal O}$ be the complement of $\cal O$ in $\A(E)$.
Let us prove $\bar{\cal O}$ is closed.
Recall $\A(E)$ is topologized by $L^2_k$ for fixed $k\geq 4$.
Suppose it is given a sequence $\{A_i\}_{i=1,2,\cdots}\subset \bar{\cal O}$ such that 
\begin{itemize}
\item $A_i\to A$ for some $A\in \A(E)$ in $L^2_k$, and
\item there exists a sequence $\{\phi_i\}\subset\Gamma(S^-)$ which satisfies $D_{A_i}^*\phi_i =0$ and $||\phi_i||_{L^2_k}=1$.
\end{itemize}
Let $a_i=A_i-A$. 
By the elliptic regularity, there exists some constant $C$ such that
$$
||\phi_i||_{L^2_{k+1}}\leq C(||\phi_i||_{L^2_k}+||D_A^*\phi_i ||_{L^2_k}) .
$$
Since $D_A^*\phi_i=-\frac12\rho(a_i)\phi_i$ and $a_i$ and $\phi_i$ is $L^2_k$-bounded, $D_A^*\phi_i$ is also $L^2_k$-bounded. 
Therefore $\phi_i$ is $L^2_{k+1}$-bounded. 
By Rellich's theorem, a subsequence of $\{\phi_i\}$ converge to some $\phi$ in $L^2_k$ with norm $1$.
%\qed
\end{proof}

Let us prove \propref{prop:U}.
\begin{proof}[Proof of \propref{prop:U}]
({\it Cf.} \cite{Moore}, \S3.4.)
The differential of $\muh$,  
$$D\muh_{(A,\Phi, \eta)}\colon\ker d\times(\ker d)^\perp\oplus\Gamma(S^+)\oplus\Omega^+(i\lambda)\to \Gamma(S^-)\oplus\Omega^+(i\lambda)$$
 is given by
$$
D\muh_{(A,\Phi,\eta)}(a,b,\phi,\sigma) = (D_A\phi+\frac12\rho(a+b)\Phi, d^+b -Dq_{\Phi}(\phi) -\sigma).
$$
Let us consider the subset ${\cal U}\subset \ker d\times\Omega^+(i\lambda)$ consisting of $(a,\eta)\in\ker d\times\Omega^+(i\lambda)$ satisfying the following property: 
$$
\text{If $b\in(\ker d)^\perp$ satisfies $F_{A_0}^+ + d^+ b = \eta$, then $D_{A_0 + a + b}$ is surjective.} 
$$
Since the restriction of $d^+$ to $(\ker d)^\perp$ is a linear homeomorphism between $(\ker d)^\perp$ and $\Omega^+(i\lambda)$ if $b_+(X;\lambda)=0$,  the space ${\cal U}$ is gauge invariant and open-dense by \lemref{lem:dense}. 
Now we claim that, if $(A,\Phi,\eta)$ is a solution to the equation $\muh=0$, the differential $D\muh_{(A,\Phi,\eta)}|_{\cal U^\prime}$ is surjective: 
We may assume $0\in{\cal U^\prime}$.
Suppose $(\psi,c)\in\Gamma(S^+)\times \Omega^+(i\lambda)$ is perpendicular to the image of $D\muh_{(A,\Phi,\eta)}|_{\cal U^\prime}$.
By holding $\phi=0$ and $a=b=0$ and varying $\sigma$, we obtain $c$ must be $0$.
If $\Phi\not\equiv 0$, then it can be proved that $\psi\equiv 0$ by the unique continuation theorem as in the proof of \lemref{lem:surj}.
If $\Phi\equiv 0$, then $\psi$ must be $0$ by the definition of $\cal U^\prime$.
%\qed
\end{proof}

By \propref{prop:U} and the implicit function theorem, 
$$
{\cal Z} = \{(A,\Phi,\eta)\in {\cal U}^\prime\,|\, \muh(A,\Phi,\eta)=0 \}
$$
is a submanifold in ${\cal U}^\prime$. 

Let us consider the projection:
$$
\tilde{\pi}\colon\ker d\times(\ker d)^\perp\times \Gamma(S^+)\times\Omega^+(i\lambda)\to\ker d\times \Omega^+(i\lambda)
.$$
Let us take a subgroup $\K_{\gamma}$ as in \subsecref{subsec:B}. 
Note that $\K_{\gamma}$ acts freely on $\ker d$, and $\ker d/\K_{\gamma}$ is isomorphic to a $b^1(X;\lambda)$-dimensional torus $T^{b_1(X;\lambda)}$.
Recall that $\cal U^\prime$ is gauge invariant.
Then, by restricting $\tilde{\pi}$ to $\cal Z$ and dividing it by $\K_{\gamma}$, we obtain a map,
$$
\pi\colon {\cal Z}/\K_{\gamma}\to T^{b_1(X;\lambda)}\times \Omega^+(i\lambda).
$$
This is a smooth map between Banach manifolds.
As in the Seiberg-Witten case, we can prove the following:
\begin{Proposition}
The map $\pi$ is a Fredholm map whose index is
$$
d^\prime = d - b^1(X;\lambda) = \ind D_{A_0}=\frac14(\tilde{c}_1(E)^2-\sign(X)).
$$
\end{Proposition}
\begin{proof}
The local slice of ${\cal K}_\gamma$-action at $(A,\Phi)$  is given by the set of elements
$$
(\alpha,\phi,\sigma)\in\Omega^1(i\lambda)\times\Gamma(S^+)\times\Omega^+(i\lambda)
$$
which are $L^2$-perpendicular to $(-2du, u\Phi,0)$ for every $u\in\Omega^0(i\lambda)$.
Let us define $f(\phi,\Phi)\in \Omega^0(i\lambda)$ by the relation 
$$
\langle \phi, u\Phi\rangle_\R = \langle f(\phi,\Phi), u\rangle_{i\lambda},
$$ 
where $\langle \cdot, \cdot \rangle_{i\lambda}$ is the natural metric on $i\lambda=i(l\otimes \R)$.
The tangent space of ${\cal Z}/{\cal K}_\gamma$ is identified with the kernel of the map
$$
F\colon \Omega^1(i\lambda)\times\Gamma(S^+)\times\Omega^+(i\lambda) \to \Omega^0(i\lambda)\times\Gamma(S^-)\times\Omega^+(\lambda)
$$
defined by 
$$
F(\alpha,\phi,\eta) = (-2 d^*\alpha + f(\phi,\Phi), D\hat{\mu}(\alpha,\phi,\sigma)).
$$
Then, it follows from  the standard argument  (e.g. \cite{Nicolaescu}, \S 1.5.2) that $\pi$ is Fredholm,  
and the index of $\pi$ is given by the sum of the index of $D_A$ and the index of the restriction of $d^*+d^+$ to the $L^2$-complement of the space of harmonic $1$-forms.
%\qed
\end{proof}

By the Sard-Smale theorem, for a generic choice of $(t,\eta)\in T^{b_1(X;\lambda)}\times \Omega^+(i\lambda)$, we obtain a $d^\prime$-dimensional manifold
$$
\M^\prime(t,\eta)=\pi^{-1}(t,\eta)\subset\CC/\K_{\gamma}.
$$

The quotient group $\G_0/\K_{\gamma}\cong \{\pm 1\}$ still acts on $\M^\prime(t,\eta)$, and there exists a unique fixed point.
Then the quotient space $\M(t,\eta)=\M^\prime(t,\eta)/\{\pm 1\}$ is a $V$-manifold which has a unique quotient singularity.
Around the singularity, we can take an open neighborhood $N$ of the form of a cone of $\RP^{d^\prime-1}$.
Removing $N$ from $\M(t,\eta)$, we obtain a $d^\prime$-dimensional compact manifold 
$$
\overline{\M}(t,\eta) = \M(t,\eta)\setminus N,
$$
whose boundary is $\RP^{d^\prime-1}$.

Now, we prove the lemma which will be a key
 point of our argument.
\begin{Lemma}\label{lem:nonposi}
If $b_+(X;l)=0$, then $d^\prime\leq 0$.
\end{Lemma}
\begin{proof}
Suppose $d^\prime=\ind D_{A_0}>0$.
We obtain a $d^\prime$-dimensional compact manifold $\overline{\M}(t,\eta)$ as above.
Note that $\overline{\M}(t,\eta)  \subset \CC^*/\G_0$ and $\partial(\overline{\M}(t,\eta))\cong\RP^{d^\prime-1}$.
Then there exists a class $A\in H^{d^\prime-1}(\CC^*/\G_0;\Z/2) = H^{d^\prime-1}(\RP^\infty\times T^{b_1(X;\lambda)};\Z/2)$ so that $\langle A, [\partial(\overline{\M}(t,\eta))]\rangle\neq 0$. 
This is a contradiction.
%\qed
\end{proof}

%
%
%%%%%%%%%%%%%%%%%%%%%%%%%%%%%%%%%%%%%%%%%%%%%%%%%%%%%%%%%%%%%%%%%%%%%%%%%%%%%%%
\section{Proof of \thmref{thm:main}}\label{sec:proof}
%%%%%%%%%%%%%%%%%%%%%%%%%%%%%%%%%%%%%%%%%%%%%%%%%%%%%%%%%%%%%%%%%%%%%%%%%%%%%%%
%
%
In this section, we complete the proof of \thmref{thm:main}.
Suppose that $X$ and $l$ satisfy the conditions in \thmref{thm:main}.
\begin{Definition}
An element $w$ in a lattice $L$ is called characteristic if $w\cdot v\equiv v\cdot v$ mod $2$  for any $v\in L$.
\end{Definition}
\begin{Lemma}\label{lem:int}
The second Stiefel-Whitney class $w_2(X)$ has a lift in $H^2(X;l)$.
Moreover, for every class $c\in H^2(X;l)$ whose class $[c]$ in $H^2(X;l)/\text{\rm torsion}$ is a characteristic element of $Q_{X,l}$, there exists a torsion class $\delta\in H^2(X;l)$ such that $c+\delta$ is a lift of $w_2(X)$.
\end{Lemma}
\begin{proof}
({\it Cf.} \cite{AL}.)
Note that $l^*=l$ and $\Hom (l;\Z/2)=\Z/2$.
By the universal coefficient theorem, we have a commutative diagram,
$$
\begin{CD}
 \Ext(H_1(X;l),\Z) @>>> H^2(X;l)@>{h_1}>>\Hom(H_2(X;l),\Z)\\
 @VV{\rho_0}V @VV{\rho_1}V @VV{\rho_2}V \\
 \Ext(H_1(X;l),\Z/2) @>{k}>> H^2(X;\Z/2)@>{h_2}>>\Hom(H_2(X;l),\Z/2).
\end{CD}
$$
Note that the homomorphisms $h_1$ and $h_2$ are given as follows:
Let $[X]\in H^4(X;\Z)$ and $[X]_2\in H^4(X;\Z/2)$ be the fundamental classes in coefficients $\Z$ and $\Z/2$.
For $a\in  H^2(X;l)$, $a^\prime\in H^2(X;\Z/2)$ and $\beta\in H_2(X;l)$, 
$$
h_1(a)(\beta) = \langle a\cup b, [X]\rangle,\quad
h_2(a^\prime)(\beta) = \langle a^\prime\cup b, [X]_2\rangle,
$$
where $b\in H^2(X;l)$ is the Poincar\'{e} dual of $\beta$, $a^\prime\cup b$ is defined by the cup product 
$$
H^2(X;\Z/2)\otimes H^2(X;l)\to H^2(X;l\otimes \Z/2)=H^2(X;\Z/2).
$$
Let $S\in \Hom(H_2(X;l),\Z/2)$ be the homomorphism defined by 
$$
S(\beta) = \langle b\cup b, [X]\rangle \mod 2,
$$
where $\beta\in H_2(X;l)$ and $b\in H^2(X;l)$ is the Poincar\'{e} dual of $\beta$.
If $c\in H^2(X;l)$ satisfies the assumption of the lemma, then $h_2\rho_1(c) = \rho_2h_1(c)=S(\gamma)$, where $\gamma$ is  the Poincar\'{e} dual of $c$.
On the other hand, $h_2(w_2(X)) = S$ by Wu's formula.
Therefore $h_2(w_2(X) - \rho_1(c))=0$, and there exists a class $\delta^\prime\in\Ext(H_1(X;l);\Z/2)$ such that  
$$
w_2(X) - \rho_1(c)=k(\delta^\prime).
$$
Since $\rho_0$ is surjective, there exists a lift $\delta\in \Ext(H_1(X;l);\Z)$ such that $\rho_0(\delta) = \delta^\prime$, and this is a required $\delta$.
%\qed
\end{proof}

\begin{Theorem}\label{thm:ineq}
Let $X$ be a closed, connected, oriented smooth $4$-manifold. 
Suppose we have a nontrivial $\Z$-bundle $l\to X$ satisfying $b_+(X;l)=0$.
Let $\lambda=l\otimes \R$.
Then, for every cohomology class $C\in H^2 (X;l)$ which satisfies $[C]_2 + w_1(\lambda)^2=w_2(X)$, where $[C]_2$ is the mod $2$ reduction of $C$, the inequality $|C^2|\geq b_2(X;l)$ holds.
\end{Theorem}
\begin{proof}
If $b_+(X;l)=0$, then $C^2\leq 0$ for $C\in H^2 (X;l)$ and $\sign(X) = -b_2(X;l)$.
For $C\in H^2 (X;l)$ satisfying the assumption, there is a $\Spin^{c_-}$-structure on $X$ whose $\OO(2)$-bundle $E$ has $\tilde{c}_1(E) =C$ by \propref{prop:spin-p}.
Let us consider the $\Pin^-(2)$-monopole moduli space on the $\Spin^{c_-}$-structure. 
Then \lemref{lem:nonposi} implies that $d^\prime=1/4(C^2-\sign(X)) =1/4(C^2 +b_2(X;l)) \leq 0$.
Thus, $|C^2|\geq b_2(X;l)$ holds.
%\qed
\end{proof}

To complete the proof of \thmref{thm:main}, we invoke the following theorem due to Elkies.
\begin{Theorem}[Elkies\cite{Elkies}]\label{thm:Elkies}
Let $L$ be a definite unimodular form over $\Z$.
If every characteristic element $w\in L$ satisfies $|w^2|\geq \rank L=n$, then $L$ is isomorphic to the standard form $\Z^n$.
\end{Theorem}
\begin{proof}[Proof of \thmref{thm:main}]
We can assume that $b_+(X;l)=0$ by reversing the orientation if necessary.
Under the assumptions of  \thmref{thm:main}, Wu's formula, \lemref{lem:int} and \thmref{thm:ineq} imply that every characteristic element $C$ of $Q_{X,l}$ satisfies $|C^2|\geq \rank Q_{X,l}$. 
Then, by Elkies' theorem, $Q_{X,l}$ should be the standard form.
%\qed
\end{proof}

%
%
%%%%%%%%%%%%%%%%%%%%%%%%%%%%%%%%%%%%%%%%%%%%%%%%%%%%%%%%%%%%%%%%%%%%%%%%%%%%%%%
\section{Proof of \thmref{thm:10/8}}\label{sec:10/8}
%%%%%%%%%%%%%%%%%%%%%%%%%%%%%%%%%%%%%%%%%%%%%%%%%%%%%%%%%%%%%%%%%%%%%%%%%%%%%%%
%
%
In this section, we prove \thmref{thm:10/8} by using the technique of the finite dimensional approximation \cite{BF} and equivariant $K$-theory as in Bryan's paper \cite{Bryan}.
We also give an alternative proof of \thmref{thm:main} by the same technique.

%%%%%%%%%%%%%%%%%%%%%%%
\subsection{The $\Pin^-(2)$-monopole map}\label{subsec:mu}
%%%%%%%%%%%%%%%%%%%%%%

Let us introduce the $\Pin^-(2)$-monopole map $\tilde{\mu}$ defined as follows ({\it Cf.} \cite{BF}, p.11):
\begin{gather*}
\begin{split}
\tilde{\mu}\colon \A(E)\times(\Gamma(S^+)\oplus &\Omega^1(X;i\lambda))\\
\to& \A(E)\times(\Gamma(S^-)\oplus \Omega^+(X;i\lambda)\oplus \Omega^0(X;i\lambda)\oplus H^1(X; i\lambda),\end{split}\\
(A,\phi,a)\mapsto (A, D_{A+a}\phi, F^+_{A} + d^+a - q(\phi), d^*a, a_{harm}),
\end{gather*}
where $a_{harm}$ is the harmonic part of $a$.
When $\tilde{c}_1(E)\neq 0$, let $\G=\G_0$ act trivially on forms.
When $\tilde{c}_1(E)= 0$, let $\G$ act on forms by multiplication of $\pm 1$ via the projection $\G\to \G/\G_0\cong \{\pm 1\}$.
Then the monopole map $\tilde{\mu}$ is $\G$-equivariant.

Let us choose a reference connection $A$ and take a subgroup $\K_{\gamma}\subset \G_0$ as in \subsecref{subsec:B}. 
The subspace $A+\ker d\subset \A(E)$ is preserved by the action of $\K_{\gamma}$, and the $\K_{\gamma}$-action is free.
The quotient space is isomorphic to the torus $T^{b_1(X;l)}=H^1(X;\lambda)/H^1(X;l)$.
Let $\VV$ and $\WW$ be the quotient spaces, 
\begin{align*}
\VV=& (A+\ker d)\times (\Gamma(S^+)\oplus \Omega^1(X;i\lambda))/\K_{\gamma},\\
\WW=&(A+\ker d) \times(\Gamma(S^-)\oplus \Omega^+(X;i\lambda)\oplus \Omega^0(X;i\lambda)\oplus H^1(X; i\lambda))/\K_{\gamma}.
\end{align*}
Then $\VV$ and $\WW$ are bundles over $T^{b_1(X;l)}$.
Dividing $\tilde{\mu}$ by $\K_{\gamma}$, we obtain a fiber preserving map
$$
\mu=\tilde{\mu}/\K_{\gamma}\colon \VV\to \WW.
$$
Then $\G_0/\K_{\gamma}=\{\pm 1\}$ still acts on $\VV$ and $\WW$, and $\mu$ is a $\Z/2$-equivariant map in general. 
If $\tilde{c}_1(E)=0$,  take a flat connection on $E\cong\underline{\R}\oplus\lambda$ as a reference connection which is the product connection of flat connections on $\underline{\R}$ and $\lambda$.
Then $\mu$ is a $\Z/4$-equivariant map.

For a fixed $k>4$, we take the fiberwise $L^2_k$-completion of $\VV$ and the fiberwise $L^2_{k-1}$-completion of $\WW$.
Then we can prove the map $\mu$ is a Fredholm proper map as in \cite{BF}.
In fact, we can readily prove the following by using the Weitzenb\"{o}ck formula \eqref{eq:Weitz}. 
\begin{Proposition}[\cite{BF}]
Preimages $\mu^{-1}(B)\subset\VV$ of bounded disk bundles $B\subset \WW$ are contained in bounded disk bundles.
\end{Proposition}
With this understood, we can construct a finite dimensional approximation $f\colon V\to W$ of $\mu$ between some finite rank vector bundles over $T^{b_1(X;l)}$ as in \cite{BF}.
The map $f$ is also a $\Z/2$(or $\Z/4$)-equivariant proper map.
\begin{Remark}
As mentioned in \remref{rem:10/8}(3), the $\Pin^-(2)$-monopole map can be identified with the $I$-invariant part of the ordinary Seiberg-Witten monopole map.
To see this, we need a little care on the gauge transformation group because the based gauge group which is used in the Seiberg-Witten monopole map is not compatible to $\cal K_\gamma$.
However, by constructing another subgroup compatible to $\cal K_\gamma$, we can obtain such an identification.
This issue will be discussed elsewhere.
\end{Remark}
\begin{Remark}
We can further develop $\Pin^-(2)$-monopole gauge theory. 
Many things in the Seiberg-Witten theory could also be considered in the $\Pin^-(2)$-monopole theory. 
Especially, we can define $\Pin^-(2)$-monopole invariants and their cohomotopy refinements. 
It would be also interesting to consider gluing formulas, Floer theory, and so on.
All of these issues are left to future researches.
\end{Remark}

%%%%%%%%%%%%%%%%%%%%%%%
\subsection{Equivariant $K$-theory}
%%%%%%%%%%%%%%%%%%%%%%
We review several facts on equivariant $K$-theory, especially, the equivariant Thom isomorphism and tom Dieck's  character formula for the $K$-theoretic degree.
We refer to the readers \S3.3 of \cite{Bryan} and tom Dieck's book \cite{Dieck}, pp.254--255.

Let $V$ and $W$ be complex $\Gamma$-representations for some compact Lie group $\Gamma$.
Let $BV$ and $BW$ be $\Gamma$-invariant balls in $V$ and $W$ and let $f\colon BV\to BW$ be a $\Gamma$-map preserving the boundaries $SV$ and $SW$.
The $K$-group  $K_\Gamma(V)$ is defined as $K_\Gamma(BV,SV)$, and the equivariant Thom isomorphism theorem says that $K_\Gamma(V)$ is a free $R(\Gamma)$-module with the Bott class $\lambda(V)$ as generator, where $R(\Gamma)$ is the complex representation ring of $\Gamma$.
The map $f$ induces a homomorphism $f^*\colon K_\Gamma(W)\to K_\Gamma(V)$.
The $K$-theoretic degree $\alpha_f\in R(\Gamma)$ is uniquely determined  by the relation $f^*(\lambda(W))=\alpha_f\cdot\lambda(V)$. 

For $g\in\Gamma$, let $V_g$ and $W_g$ be the subspaces of $V$ and $W$ fixed by $g$, and let $V_g^{\bot}$ and $W_g^{\bot}$ be their orthogonal complements.
Let $f^g\colon V_g\to W_g$ be the restriction of $f$, and let $d(f^g)$ be the ordinary topological degree of $f^g$.
(Note that $d(f^g)=0$ if $\dim V_g\neq \dim W_g$.)
For $\beta\in R(\Gamma)$, let $\Lambda_{-1}\beta$ be the alternating sum $\sum(-1)^i\Lambda^i\beta$ of exterior powers.

Then tom Dieck's character formula \cite{Dieck} is,
\begin{equation}\label{eq:Dieck}
\tr_g(\alpha_f)=d(f^g)\tr_g(\Lambda_{-1}(W_g^{\bot} - V_g^{\bot})),
\end{equation}
where $\tr_g$ is the trace of the $g$-action.

%%%%%%%%%%%%%%%%%%%%%%%
\subsection{Proof of \thmref{thm:10/8}}
%%%%%%%%%%%%%%%%%%%%%%
Suppose $X$  and a $\Z$-bundle $l$ satisfy the assumptions of \thmref{thm:10/8}.
Let $\lambda=l\otimes \R$ and $E=\underline{\R}\oplus\lambda$.
By the assumptions, a $\Spin^{c_-}$-structure $(P,\tau)$ for $(X,E)$ exists by \propref{prop:spin-p}, and we obtain a finite dimensional approximation $f\colon V\to W$ of the $\Pin^-(2)$-monopole map on $(P,\tau)$.
Since  $\tilde{c}_1(E)=0$, $f$ is a $\Gamma=\Z/4$-equivariant proper map.
If $b_1(X;l)>0$, by restricting $f$ to the fiber over the origin of $T^{b_1(X;l)}$ which is represented by the fixed reference connection $A$, $f$ can be assumed to be a $\Gamma$-map between (real) $\Gamma$-representation $V$ and $W$.
In fact, $f$ can be considered as a map of the following form,
$$
f\colon\tilde{\R}^m \oplus \C_1^{n+k}\to \tilde{\R}^{m+b} \oplus \C_1^{n},
$$
where $\Gamma=\Z/4$ acts on $\tilde{\R}$ by multiplication of $\pm 1$ via the surjection $\Z/4\to\{\pm 1\}$, and on $\C_k$ by multiplication of $g=\exp 2\pi\sqrt{-1}k/4$ for some fixed generator $g$ of $\Gamma$, $m,n$ are some positive integers, $b=b_+(X;l)$ and 
$$
k = \frac12\ind_\R D_A = \frac18(\tilde{c}_1(E)^2-\sign(X)) = - \frac{1}8\sign(X).
$$ 
As in \cite{Furuta}, take the complexification of $f$ as $f(u\otimes 1+v\otimes i) = f(u)\otimes 1 + f(v)\otimes i$.
Now the complexified $f$ is of the form,
$$
f\colon \tilde{\C}^m \oplus (\C_1\oplus \C_{-1})^{n+k}\to  \tilde{\C}^{m+b} \oplus (\C_1\oplus \C_{-1})^{n},
$$
where $\tilde{\C}=\tilde{\R}\otimes\C$. 
Let us apply tom Dieck's formula \eqref{eq:Dieck}.
Since $V_g=W_g=0$, $d(f^g)=1$. 
Then we have,
$$
\tr_g(\alpha_f) = \tr_g(\Lambda_{-1}(\tilde{\C}^b - (\C_1\oplus \C_{-1})^{k}) = \tr_g((\C-\tilde{\C})^b(2\C-\C_1\oplus \C_{-1})^{-k}) = 2^{b-k}.
$$
Since $\tr_g(\alpha_f)$ is an integer, we have $b-k\geq 0$.
Thus, \thmref{thm:10/8} is proved.
\begin{Remark}
In the proof of \thmref{thm:10/8}, we restrict the finite dimensional approximation $f$ to a fiber, and take the complexification of it.
Due to such modifications of $f$, the inequality we obtained might be somewhat weaker than expected.
One could improve the inequality by using the technique of \cite{FK}.
\end{Remark}

%%%%%%%%%%%%%%%%%%%%%%%
\subsection{An alternative proof of \thmref{thm:main}}
%%%%%%%%%%%%%%%%%%%%%%
In this subsection, we give an alternative proof of \thmref{thm:main} by giving an alternative proof of \lemref{lem:nonposi}.
Suppose $X$ and $l$ satisfy the assumption of \thmref{thm:main}. 
We may assume $b_+(X;l)=0$ by reversing the orientation of $X$ if necessary.
Let $E$ be an $\OO(2)$-bundle such that $\det E=\lambda$, and suppose a $\Spin^{c_-}$-structure on $(X,E)$ is given.
Then we have a $\Gamma=\Z/2$-equivariant finite dimensional approximation $f\colon V\to W$ of the $\Pin^-(2)$-monopole map.
By restricting $f$ to a fiber if $b_1(X;l)>0$, we may assume $f$ has the form of
$$
f\colon \R^m\oplus \tilde{\C}^n\to  \R^m\oplus \tilde{\C}^{n+k},
$$
where $\Gamma\cong \{\pm 1\}$ acts on $\R$ trivially, and on $\tilde{\C}$ by multiplication of $\pm 1$, and $m,n$ are some positive integers, and 
$$
k = -\frac12\ind_\R D_A = - \frac18(\tilde{c}_1(E)^2-\sign(X)).
$$  
Take the complexification of $f$ and apply tom Dieck's formula \eqref{eq:Dieck} for $g=-1$.
Then,
$$
\tr_g(\alpha_f) = \tr_g((\C-\tilde{\C})^{2k}) = 2^{2k}.
$$
Therefore $k\geq 0$, and \lemref{lem:nonposi} is proved.


\begin{thebibliography}{99}
\bibitem{AL} D.~Acosta and T.~Lawson,
{\it Even non-spin manifolds, ${\rm spin}^c$ structures, and duality},
Enseign. Math. {\bf 43} (1997), no. 1-2, 27--32. 


\bibitem{AB} M.~F.~Atiyah and R.~Bott,
{\it A Lefschetz fixed point formula for elliptic complexes: \Romnum1},
Ann. of Math. {\bf 86}, 374--407.

\bibitem{BF} S.~Bauer and M.~Furuta, 
{\it A stable cohomotopy refinement of Seiberg-Witten invariants. \Romnum1},
Invent. Math. {\bf 155} (2004), no. 1, 1--19. 

\bibitem{Bohr} C.~Bohr, 
{\it On the signatures of even 4-manifolds},
Math. Proc. Cambridge Philos. Soc. {\bf 132} (2002), no. 3, 453--469. 

\bibitem{Bryan} J.~Bryan,
{\it Seiberg-Witten theory and $\Z/2^p$ actions on spin $4$-manifolds},
Math. Res. Lett. {\bf 5} (1998),  165--183. 

\bibitem{Dieck}T.~tom Dieck,
Transformation groups and representation theory,
Lecture Notes in Mathematics, 766. Springer, Berlin, 1979. 


\bibitem{D1} S.~K.~Donaldson,
{\it An application of gauge theory to four dimensional topology},
J. Diff. Geom. {\bf 18} (1983), 279--315.

\bibitem{D2} S.~K.~Donaldson,
{\it The orientation of Yang-Mills moduli spaces and $4$-dimensional topology},
J. Diff. Geom. {\bf 26} (1987), 397--428.

\bibitem{Elkies} N.~D.~Elkies,
{\it A characterization of the $\Z^n$ lattice}, 
Math. Res. Lett. {\bf 2} (1995), no. 3, 321--326. 

\bibitem{FQ} M.~H.~Freedman and F.~Quinn,
Topology of 4-manifolds,
Princeton Mathematical Series, 39. Princeton University Press, Princeton, NJ, 1990.

\bibitem{Froyshov}K.~A. ~Froyshov,
{\it $4$-manifolds and intersection forms with local coefficients},
preprint, arXiv:1004.0077.

\bibitem{FurutaG}M.~Furuta,
{\it A remark on a fixed point of finite group action on $S^4$},
Topology {\bf 28} (1989), no. 1, 35--38. 

\bibitem{Furuta}M.~Furuta,
{\it Monopole equation and the $\frac{11}8$-conjecture},
Math. Res. Lett. {\bf 8} (2001), no. 3, 279--291. 

\bibitem{FK} M.~Furuta and Y.~Kametani,
{Equivariant maps between sphere bundles over tori and $KO$-degree},
preprint, arXiv:math/0502511

\bibitem{HK} I.~Hambleton and M.~Kreck,
{\it Smooth structures on algebraic surfaces with cyclic fundamental group},
Invent. Math. {\bf 91} (1988), no. 1, 53--59.

\bibitem{Kirby} R.~Kirby,
Topology of $4$-manifolds,
Lecture Notes in Mathematics, 1374. Springer-Verlag, Berlin, 1989. 

\bibitem{LL} R.~Lee and T.-J.~Li,
{\it Intersection forms of non-spin four manifolds}.
Math. Ann. {\bf 319} (2001), no. 2, 311--318. 

\bibitem{Moore} J.~D.~Moore,
{Lectures on Seiberg-Witten invariants},
Lecture Notes in Mathematics, 1629. Springer-Verlag, Berlin, 1996.

\bibitem{Morgan} J.~W.~Morgan,
{The Seiberg-Witten equations and application to the topology of smooth four-manifolds},
Mathematical Notes, Princeton Univ. Press, 1996.

\bibitem{Nfree} N.~Nakamura,
{\it A free $\Z_p$-action and the Seiberg-Witten invariants},
J. Korean Math. Soc. {\bf 39} (2002), no. 1, 103--117. 

\bibitem{Nicolaescu} L.~I.~Nicolaescu,
Notes on Seiberg-Witten theory, 
Graduate Studies in Mathematics, 28. American Mathematical Society, Providence, RI, 2000.

\bibitem{Switzer} R.~M.~Switzer,
Algebraic topology—homotopy and homology. 
Die Grundlehren der mathematischen Wissenschaften, Band 212. Springer-Verlag, New York-Heidelberg, 1975.


\bibitem{TW} G.~Tian and S.~Wang,
{\it Orientability and real Seiberg-Witten invariants},
Internat. J. Math. {\bf 20} (2009), no. 5, 573–604. 



\end{thebibliography}
\end{document}